\documentclass[12pt]{article}

\usepackage{amsmath,amssymb,amsthm,mathtools,amstext}
\usepackage{float}

\usepackage{graphicx}
\usepackage{subfig}
\usepackage{subcaption}
\usepackage{calc}
\usepackage{tikz}  
\usetikzlibrary{shapes.geometric, arrows, arrows.meta, positioning} 

\usepackage{pgfplots}
\pgfplotsset{compat=1.17}
\usepackage{pgf-pie}
\usepackage{svg}

\usepackage[ruled,vlined]{algorithm2e}
\usepackage{algcompatible}

\usepackage{xcolor}
\usepackage{titlesec}

\usepackage[round,authoryear]{natbib}
\usepackage{verbatim}
\usepackage{chapterbib}
\usepackage{hyperref}
\usepackage{cancel}
\usepackage{multirow}
\usepackage{multicol}
\usepackage{eurosym}
\usepackage{academicons}
\usepackage[font=small]{caption}

\usepackage[left]{lineno}  

\newtheorem{proposition}{Proposition}[section]

\newtheorem{theorem}{Theorem}
\newtheorem{assumption}[proposition]{Assumption}
\newtheorem{lemma}[proposition]{Lemma}

\theoremstyle{definition}
\newtheorem{definition}[proposition]{Definition}
\newtheorem{remark}[proposition]{Remark}

\tikzstyle{process} = [
  rectangle,
  minimum width=1cm,
  minimum height=1cm,
  text centered,
  draw=black,
  fill=blue!15,
  align=center  
]

\tikzstyle{arrow} = [thick,->,>=stealth]

\begin{document}
\usetikzlibrary{calc}


Highlights
\begin{itemize}
\item Proof that the Maximum-Likelihood Estimator in the Linkage Model is consistent.
\item Proof that the Maximum-Likelihood Estimator in the Linkage Model is asymptotically normally distributed.
\item Statistical Test to decide whether  the Linkage Model or the Admixture Model fits better to the data.
\end{itemize}

\newpage

\title{\LARGE Statistical Test to compare the Linkage Model and the Admixture Model based on Central Limit Results}

\maketitle

\begin{center}
   {Carola Sophia Heinzel}

Department of Mathematical Stochastics,\\
Ernst-Zermelo Straße 1,
            Freiburg im Breisgau,
            79140,Germany \\
            carola.heinzel@stochastik.uni-freiburg.de
\end{center}


\begin{abstract}
In the Admixture Model, the probability that an individual carries a certain allele at a specific marker depends on the allele frequencies in $K$ ancestral populations and the proportion of the individual's genome originating from these populations. The markers are assumed to be independent. The Linkage Model is a Hidden Markov Model that extends the Admixture Model by incorporating linkage between neighboring loci.

We prove consistency and asymptotic normality of maximum likelihood estimators for the ancestry of individuals in the Linkage Model, complementing earlier results by \citep{pfaff2004information, pfaffelhuber2022central, HEINZEL2025} for the Admixture Model. These results are used to prove that a statistical test that allows for model selection between the Admixture Model and the Linkage Model is an asymptotic level-$\alpha$-test. Finally, we demonstrate the practical relevance of our results by applying the test to real-world data from \cite{10002015global}.
\end{abstract}

\textit{
Keywords: Linkage Model, Admixture Model, Central Limit Results,  Maximum Likelihood Estimator,  Consistency, Statistical Test, Model Selection}

\section{Introduction}
The Linkage Model \citep{falush2003} is widely used to explain the genetic data of individuals. It assumes that genetic data can be described by the ancestry proportions $q$ of an individual from $K$ ancestral populations and the allele frequencies in these populations.  Furthermore, there exists a parameter $r$, which can be interpreted as the number of generations since an admixture event. The genetic distance between loci is also considered by using a Hidden Markov Model (HMM). The Admixture Model can be seen as a special case of the Linkage Model with $r = \infty$, i.e. the data is assumed to be independent across markers.

A natural question that arises is which model, the Admixture Model or the Linkage Model, fits a given dataset better. This leads to a nested model selection problem as described by e.g. \cite{anderson2004model}. In our case, we consider the statistical hypothesis test
\begin{align}
    H_0: r = \infty \quad \text{vs.} \quad H_1: r \in [0,\infty). \label{def:st}
\end{align}
This is a classical test problem, for which asymptotic theory exists \citep{wilks1938large}. However, the test is only valid as an asymptotic level-$\alpha$ test if the MLEs for both ancestry $q$ and $r$ are asymptotically normally distributed under the Linkage Model. Proving asymptotic normality requires results from the theory of HMMs.

Consistency of Maximum-Likelihood Estimators (MLEs) in HMMs with finite state and observation spaces has been addressed in foundational work by \cite{baum1966statistical} and \cite{petrie1969probabilistic}. More recent work has relaxed several of the original assumptions, e.g., \citep{leroux1992maximum, douc2004asymptotic, douc2011consistency, genon2006leroux, le2000basic, le2000exponential}, but these results assume time-homogeneous Markov chains. Central limit theorems for MLEs in HMMs have also been developed \citep{bickel1998asymptotic, douc2004asymptotic, jensen1999asymptotic, brouste2010asymptotic}. Overviews of the statistical theory of HMMs can be found in \cite{ephraim2002hidden, cappé2005inference}. However, the specific case of a time-inhomogeneous Markov chain that maintains the same stationary distribution for the hidden chain across all time points - as occurs in the Linkage Model when $q$ is the initial distribution - has not yet been studied.

The question whether the data should be used in STRUCTURE \citep{pritchard2000}(a software that estimates the ancestry and the allele frequencies for the Admixture Model or the Linkage Model) analysis has already been considered. For example, to decide whether a pair of loci is suitable for STRUCTURE, \cite{kaeuffer2007detecting} suggested using $r_{LD}$, a measure of linkage disequilibrium introduced by \cite{hill1968linkage}. They ran STRUCTURE with the Linkage Model and used logistic regression to assess the impact of $r_{LD}$ on the detection of population structure. 

To achieve our goal of establishing the theoretical properties of the test in \eqref{def:st}, we proceed as follows:  
First, we precisely define both the Admixture Model and the Linkage Model. Next, we prove the asymptotic normality of the MLEs in the Linkage Model, both when the number of observations tends to infinity. Based on these results, we quantify the uncertainty of the MLEs and we construct a statistical test for \eqref{def:st} and prove that it is an asymptotic level-$\alpha$ test. Finally, we evaluate the statistical test through simulations and apply it to real data from \cite{10002015global}.

\section{Models}

We first define the Linkage Model for one individual and bi-allelic markers in Definition \ref{def:lm}. The number of alleles on chromosome $c \in \{1,..., C\}$ at marker $m \in \{1,..., M_c\}$ is called $X_{c,m} \in \{0, 1\}.$ We denote the genetic distance in centi Morgan (cM) between the loci $m-1$ and $m$ on chromosome $c$ by $d_{c,m}$. Recall that one centi Morgan corresponds to a $0.01$ crossing over between two genetic markers per Meiosis. Let $\mathbb{S}^K$ be the $(K-1)$-dimensional simplex and let $q :=(q_{1}, \ldots, q_{K  }) \in \mathbb S^K$ be the ancestries of the individual from population $1,..., K.$ The frequency of an allele in population $k \in \{1,..., K\}$ at marker $m$, on chromosome $c$ is called $p_{c, k,m}.$  {We assume in this work that the allele frequencies are known, which is called supervised setting.}
The random variable $Z_{c,m}, m = 1,..., M_c, c = 1,..., C$ names the ancestral population of the allele at marker $m$ on chromosome $c$. We write $(q^0, r^0)$ for the true values of $(q, r)$ and $\mathbb E, \mathbb P$ for the expected value and the probability respectively, with respect to the true parameters $q^0, r^0.$

\begin{definition}[Linkage Model for Haploid Individuals]\label{def:lm}
We define the Markov chain 
    \begin{align} 
     \mathbb P_{q,r}(Z_{c,1} = k) &= q_{k}, \notag \\
        \mathbb P_{q,r}(Z_{c,m} = \tilde k|Z_{c,m-1} =  k)  &=  \begin{cases}
            e^{- d_{c,m} r} + \left(1-e^{-d_{c,m} r}\right) q_{k}, \textup{ if } k = \tilde k \\
            \left(1-e^{-d_{c,m} r}\right) q_{\tilde k}, \textup{ else}, \label{eq:tp}
        \end{cases}\\
            \mathbb P_{q,r}(Z_{c+1,1} = k|Z_{c,M_c}) &= q_{k} \notag.
    \end{align}
The emission probability is defined by
\begin{align*}
    \mathbb P_{q,r}(X_{c,m} = x|Z_{c,m} = k) =  {p_{c,k,m}^x (1-p_{c,k,m})^{1-x}}.
\end{align*}
We define the log-likelihood, for $M_{total} := \sum_{c = 1}^C {M_c},$
\begin{align*}
    \ell((x_{1,1},..., x_{C,M}), (q,r)) := \frac{\log\left( \mathbb P_{q,r}((X_{c,m} = x_{c,m})_{c=1,..., C, m = 1,..., M_c}))\right)}{M_{total}}.
\end{align*}

\end{definition}

A visualization of the model is shown in Figure \ref{fig:hmm}.

\begin{figure}[h!]
\centering
\hbox{%
  \raisebox{-47mm}[0pt][2cm]{%
\begin{minipage}{0.35\textwidth}
    \centering
    \textup{(I) Reference Data Set\\
    \hphantom{ykvnavn}}
    \vspace*{0.5cm}
\includegraphics[width=\linewidth]{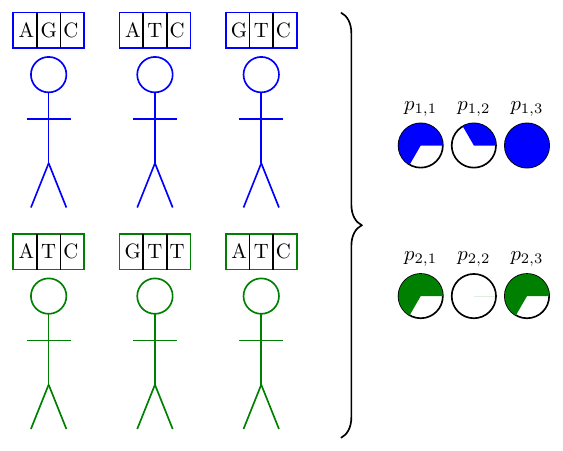} 
\end{minipage}}}
\hfill
\begin{minipage}{0.6\textwidth}
    \centering
    \textup{(II) Model}
    \vspace*{1.5cm}
\begin{tikzpicture}[
    scale=0.7,  
    every node/.style={transform shape}, 
    state/.style={circle, draw, minimum size=1.2cm, thick},
    >=Stealth,
    node distance=2.5cm and 2cm
]

\tikzset{
  triState/.style={
    draw=black,
    shape=regular polygon,
    regular polygon sides=3,
    minimum size=1.4cm,
    text height=1.5ex,
    text depth=.25ex,
    inner sep=0pt
  }
}

\newcommand{\stackedbar}{
        \begin{tikzpicture}
            \begin{axis}[
                ybar stacked,
                ymin=0, ymax=1,
                symbolic x coords={Ind.},
                xtick=data,
                 ytick={0,1}, 
                bar width=3pt,
                ylabel={$q_{\cdot}$},
                ylabel style={xshift=0pt, yshift =-11pt},
                width=2cm, height=2.5cm
            ]
                \addplot[fill=green!50!black] coordinates {(Ind., 0.7)};
                \addplot[fill=blue] coordinates {(Ind., 0.3)};
            \end{axis}
        \end{tikzpicture}

}

\node[triState] (z1) at (-0.5,2) {$Z_1$};
\node[triState, right=3cm of z1] (z2) {$Z_2$};
\node[triState, right=3cm of z2] (z3) {$Z_3$};

\node[draw, thick, minimum width=10cm, minimum height=1cm, below=2.6cm of z2, anchor=north] (Xrect) {};

\path (z1) ++(0,-3) coordinate (x1);
\path (z2) ++(0,-3) coordinate (x2);
\path (z3) ++(0,-3) coordinate (x3);

\node at (-0.5, -1.5) {$X_1$};
\node at (3.25, -1.5) {$X_2$};
\node at (7, -1.5) {$X_3$};

\draw[thick] (1.625, -1.95) -- (1.625,-0.95);
\draw[thick] (4.875, -1.95) -- (4.875,-0.95);

\draw[->, thick] (z1) -- node[below] {$A_{Z_1,Z_2}$} (z2);
    \node[above=50pt] (toy1) at ($(z1)!0.25!(z2)$) {$r$};
  \node[above=20pt] (bar1) at ($(z1)!0.55!(z2)$) {\stackedbar};
      \draw[->, thick] (bar1) to node[above,sloped]{\scriptsize } ($(z1)!0.55!(z2)$);
\draw[->, thick] (z2) -- node[below] {$A_{Z_2,Z_3}$} (z3);
    \node[above=50pt] (toy1) at ($(z2)!0.2!(z3)$) {$r$};
  \node[above=20pt] (bar2) at ($(z2)!0.5!(z3)$) {\stackedbar};
\draw[->, thick] (bar2) to node[above,sloped]{\scriptsize } ($(z2)!0.5!(z3)$);
\draw[->, thick] (z1) -- node[right=2pt, midway, font=\footnotesize] {$B_{Z_1,X_1}$} (x1);

\def\radius{0.2} 
\coordinate (leftCircles) at ($(z1)+(-0.8,-1.5)$);

\draw[thick] (leftCircles) ++(0,0.25) circle (\radius);
\fill[blue] (leftCircles) ++(0,0.25) -- ++(\radius,0) arc (0:240:\radius) -- cycle;

\draw[thick] (leftCircles) ++(0,-0.25) circle (\radius);
\fill[green!50!black] (leftCircles) ++(0,-0.25) -- ++(\radius,0) arc (0:240:\radius) -- cycle;

\draw[->, thick]
  ($(leftCircles)+(0.17,0.05)$) -- ($(z1)+(0,-1.455)$);

\draw[->, thick] (z2) -- node[right=2pt, midway, font=\footnotesize] {$B_{Z_2,X_2}$} (x2);
\draw[->, thick] (z3) -- node[right=2pt, midway, font=\footnotesize] {$B_{Z_3,X_3}$} (x3);

\coordinate (leftCircles) at ($(z2)+(-0.8,-1.5)$);

\draw[thick] (leftCircles) ++(0,0.25) circle (\radius);
\fill[blue] (leftCircles) ++(0,0.25) -- ++(\radius,0) arc (0:120:\radius) -- cycle;
\draw[->, thick]
  ($(leftCircles)+(0.17,0.05)$) -- ($(z2)+(0,-1.455)$);
\draw[thick] (leftCircles) ++(0,-0.25) circle (\radius);
\fill[green!50!black] (leftCircles) ++(0,-0.25) -- ++(\radius,0) arc (0:0:\radius) -- cycle;

\coordinate (leftCircles) at ($(z3)+(-0.8,-1.5)$);
\draw[->, thick]
  ($(leftCircles)+(0.17,0.05)$) -- ($(z3)+(0,-1.455)$);
\draw[thick] (leftCircles) ++(0,0.25) circle (\radius);
\fill[blue] (leftCircles) ++(0,0.25) -- ++(\radius,0) arc (0:360:\radius) -- cycle;

\draw[thick] (leftCircles) ++(0,-0.25) circle (\radius);
\fill[green!50!black] (leftCircles) ++(0,-0.25) -- ++(\radius,0) arc (0:240:\radius) -- cycle;

\node[font=\footnotesize, rotate=90] at (-2,2.3) {\textbf{Hidden}};\node[font=\footnotesize, rotate = 90]  at (-2,-1.5) {\textbf{Data}};

\end{tikzpicture}
\end{minipage}
\begin{tikzpicture}[overlay]
    \node[anchor=south, yshift=1cm] 
        at (-12.3, -4) {(III) Estimators};
  \node (a) at (-14,-2) {};
  \node (b) at (0,-2) {};
  \draw [decorate,decoration={brace,amplitude=10pt,mirror}]
    (a) -- (b) node[midway,below=10pt]{$\Downarrow$};
    \node at (-8.5, -3.8){$\hat r$};
    \node at (-5.5, -4) {
        \begin{tikzpicture}
            \begin{axis}[
                ybar stacked,
                ymin=0, ymax=1,
                symbolic x coords={Ind.},
                xtick=data,
                 ytick={0,0.5, 1}, 
                bar width=5pt,
                ylabel={$\hat q_{\cdot}$},
                width=2.5cm, height=3cm
            ]
                \addplot[fill=green!50!black] coordinates {(Ind., 0.6)};
                \addplot[fill=blue] coordinates {(Ind., 0.4)};
            \end{axis}
        \end{tikzpicture}
    };

\end{tikzpicture}
\vspace{3cm}
\caption{ {Visualization of the Linkage Model for two populations (blue and green) and three markers ($M = 3$). (I): Reference Data Set that is used to calculate the allele frequencies. In this example we only consider bi-allelic markers. The letters A, C, G and T stand for the bases of the DNA, i.e. for Adenin, Cytosin, Guanin and Thymin. From this reference data set, we receive the allele frequencies, called $p_{k,m}, k \in \{1,2\}, m \in \{1,2,3\}.$ (II): Hidden Markov Model to describe the probabilities of the observed data $X_1,X_2,X_3.$ The hidden chain is represented as $Z_1, Z_2$ and $Z_3$, the transition probabilities of state $Z_i$ to $Z_{i+1}$ are denoted as $A_{Z_i, Z_{i+1}}$ and the emission probabilities $\mathbb{P}(X_i|Z_i)$ are called $B_{Z_i, X_i}.$ Ind. stands for Individual. (III): Estimators based on the reference data set and the Linkage Model.}  }\label{fig:hmm}
\end{figure}

We always assume that the allele frequencies and $(d_{c,m})_{c = 1, ..., C, m = 1,..., M_c}$ are known. Definition \ref{def:lm} is for haploid individuals, which can easily be extended to the diploid case. 

\begin{remark}[Diploid Case]
Let $Z^j := (Z^j_{c,1},..., Z^j_{c,M}), j = 1, 2,$ be two independent Markov chains with transition matrix \eqref{eq:tp}. Let $X_{c,m}^{dip}$ be the number of alleles on chromosome $c,$ marker $m,$ for diploid individuals. The emission probabilities for the diploid case are defined by
\begin{align*}
& { \mathbb P_{q,r}\left(X^{dip}_{c,m} = x|Z^{1}_{c,m} = z_1, Z^{2}_{c,m} = z_2\right)}\\
    &\hspace{1cm} { =}\begin{cases}
        { p_{c,z_1,m} p_{c,z_2,m},} &x = 2 \\
       { p_{c,z_1,m} (1-p_{c,z_2,m}) + p_{c,z_2,m}(1-p_{c,z_1,m} ),} &x = 1 \\ 
        {  (1 - p_{c,z_1,m}) (1 - p_{c,z_2,m})},& x = 0.
    \end{cases}
\end{align*}
For the diploid case, the information about the maternal and the paternal copies \citep{choi2018comparison}, i.e. phased data, is important. There are two kinds of approaches to receive phased data: laboratory-based methods \citep{zheng2016haplotyping, amini2014haplotype, duitama2012fosmid} and computational methods \citep{choi2018comparison}. Examples for the latter type are \citep{snyder2015haplotype, loh2016reference, delaneau2012linear}.  \cite{falush2003} also proposed a method to deal with linkage without phasing. However, in this study, we assume that we have phased data as e.g. in \cite{10002015global}.

\end{remark}

While phased data is important for the Linkage Model, the Admixture Model can deal with unphased data without any problems. It is defined in Definition \ref{def:AM}.

\begin{definition}[Admixture Model]\label{def:AM}
   The Admixture Model is a special case of the Linkage Model with  $r = \infty.$ Especially, this means that it holds
   $$X_{c,m} \sim \textup{Bin}(2, \langle q_{\cdot}^0, p_{c, \cdot, m}\rangle).$$
\end{definition}
A visualization of the Admixture Model is shown in Figure \ref{fig:am}.
\begin{figure}[h!]
\centering
\hbox{%
  \raisebox{-47mm}[0pt][2cm]{%
\begin{minipage}{0.35\textwidth}
    \centering
    \textup{(I) Reference Data Set\\
    \hphantom{ykvnavn}}
    \vspace*{0.5cm}
\includegraphics[width=\linewidth]{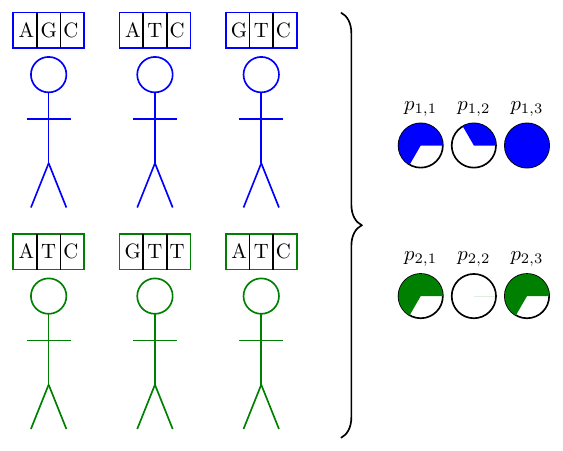} 
\end{minipage}}}
\hfill
\begin{minipage}{0.6\textwidth}
    \centering
    \textup{(II) Model\\
    \hphantom{ykvnavn}}
    \vspace*{1.5cm}
\begin{tikzpicture}[
    scale=0.7,  
    every node/.style={transform shape}, 
    state/.style={circle, draw, minimum size=1.2cm, thick},
    >=Stealth,
    node distance=2.5cm and 2cm
]

\newcommand{\stackedbar}{
        \begin{tikzpicture}
            \begin{axis}[
                ybar stacked,
                ymin=0, ymax=1,
                symbolic x coords={Ind.},
                xtick=data,
                 ytick={0,1}, 
                bar width=3pt,
                ylabel={$q_{\cdot}$},
                ylabel style={xshift=0pt, yshift =-11pt},
                width=2cm, height=2.5cm
            ]
                \addplot[fill=green!50!black] coordinates {(Ind., 0.7)};
                \addplot[fill=blue] coordinates {(Ind., 0.3)};
            \end{axis}
        \end{tikzpicture}

}

\node (z1) at (0, 2) {\stackedbar};
\node[right=2cm of z1] (z2) {\stackedbar};
\node[right=2cm of z2] (z3){\stackedbar};

\node[draw, thick, minimum width=11cm, minimum height=1cm, below=1.5cm of z2, anchor=north] (Xrect) {};

\path (z1) ++(0,-2.5) coordinate (x1);
\path (z2) ++(0,-2.5) coordinate (x2);
\path (z3) ++(0,-2.5) coordinate (x3);

\node at (0, -1.1) {$X_{i1}$};
\node at (3.5, -1.1) {$X_{i2}$};
\node at (7.5, -1.1) {$X_{i3}$};

\draw[thick] (1.625, -1.55) -- (1.625,-0.55);
\draw[thick] (5.275, -1.55) -- (5.275,-0.55);

\draw[->, thick] (z1) -- node[right=2pt, midway, font=\footnotesize] {$\langle q_{i\cdot}, p_{\cdot,1} \rangle$} (x1);

\def\radius{0.2} 
\coordinate (leftCircles) at ($(z1)+(-0.8,-1.5)$);

\draw[thick] (leftCircles) ++(0,0.25) circle (\radius);
\fill[blue] (leftCircles) ++(0,0.25) -- ++(\radius,0) arc (0:240:\radius) -- cycle;

\draw[thick] (leftCircles) ++(0,-0.25) circle (\radius);
\fill[green!50!black] (leftCircles) ++(0,-0.25) -- ++(\radius,0) arc (0:240:\radius) -- cycle;

\draw[->, thick]
  ($(leftCircles)+(0.17,0.05)$) -- ($(z1)+(0,-1.455)$);

\draw[->, thick] (z2) -- node[right=2pt, midway, font=\footnotesize] {$\langle q_{i\cdot}, p_{\cdot,2} \rangle$} (x2);
\draw[->, thick] (z3) -- node[right=2pt, midway, font=\footnotesize] {$\langle q_{i\cdot}, p_{\cdot,3} \rangle$} (x3);

\coordinate (leftCircles) at ($(z2)+(-0.8,-1.5)$);

\draw[thick] (leftCircles) ++(0,0.25) circle (\radius);
\fill[blue] (leftCircles) ++(0,0.25) -- ++(\radius,0) arc (0:120:\radius) -- cycle;
\draw[->, thick]
  ($(leftCircles)+(0.17,0.05)$) -- ($(z2)+(0,-1.455)$);
\draw[thick] (leftCircles) ++(0,-0.25) circle (\radius);
\fill[green!50!black] (leftCircles) ++(0,-0.25) -- ++(\radius,0) arc (0:0:\radius) -- cycle;

\coordinate (leftCircles) at ($(z3)+(-0.8,-1.5)$);
\draw[->, thick]
  ($(leftCircles)+(0.17,0.05)$) -- ($(z3)+(0,-1.455)$);
\draw[thick] (leftCircles) ++(0,0.25) circle (\radius);
\fill[blue] (leftCircles) ++(0,0.25) -- ++(\radius,0) arc (0:360:\radius) -- cycle;

\draw[thick] (leftCircles) ++(0,-0.25) circle (\radius);
\fill[green!50!black] (leftCircles) ++(0,-0.25) -- ++(\radius,0) arc (0:240:\radius) -- cycle;

\node[font=\footnotesize, , rotate=90]  at (-2.7,-1.1) {\textbf{Data}};

\end{tikzpicture}
\end{minipage}
\begin{tikzpicture}[overlay]
    \node[anchor=south, yshift=1cm] 
        at (-12.3, -4) {(III) Estimators};
  \node (a) at (-14,-2) {};
  \node (b) at (0,-2) {};
  \draw [decorate,decoration={brace,amplitude=10pt,mirror}]
    (a) -- (b) node[midway,below=10pt]{$\Downarrow$};
    \node at (-6.5, -4) {
        \begin{tikzpicture}
            \begin{axis}[
                ybar stacked,
                ymin=0, ymax=1,
                symbolic x coords={Ind.},
                xtick=data,
                 ytick={0,0.5, 1}, 
                bar width=5pt,
                ylabel={$\hat q_{\cdot}$},
                width=2.5cm, height=3cm
            ]
                \addplot[fill=green!50!black] coordinates {(Ind., 0.6)};
                \addplot[fill=blue] coordinates {(Ind., 0.4)};
            \end{axis}
        \end{tikzpicture}
    };

\end{tikzpicture}
\vspace{3cm}
\caption{ {Visualization of the Admixture Model for two populations (blue and green) and three markers. (I): Reference Data Set that is used to calculate the allele frequencies. In this example we only consider bi-allelic markers. (II): Visualization of the model (III): Estimators based on the reference data set and the Admixture Model.}  }\label{fig:am}
\end{figure}


Extending the models to a general number of individuals and markers with arbitrary number of alleles is straightforward. 

Finally, we define the statistical test that corresponds to test problem \ref{eq:tp}.

\begin{definition}[Statistical Test]\label{def:test}
    We define the test statistic for the test \eqref{def:st} by
    $$\Lambda := -2 \ln\left(\frac{\max\left\{ {\mathbb P_{q,r}}((X_{1,1},..., X_{C,M}), (q,\infty)): q \in \mathbb S^K\right\}}{\max\{  {\mathbb P_{q,r}}((X_{1,1},..., X_{C,M}), (q,r)): (q,r) \in \Theta\}}\right).$$ Furthermore, let $\chi^2_{1-2\alpha}$ be the $1-2\alpha$-quantile of the $\chi^2(1)$-distribution.
We reject $H_0$, if $\Lambda > \chi^2_{1-2\alpha}.$
\end{definition}

The test in Definition \ref{def:st} compares the relative support of the Linkage Model and the Admixture Model for the observed data. Importantly, this is a model-selection test: it implicitly assumes that one of the two models is adequate. It therefore does not address the possibility that both models provide a poor description of the data, in which case selecting the better model may still yield a misleading fit. This, however, is a different but widely considered research question as e.g. considered by \citep{tedeschi2006assessment, mimno2015posterior, carstens2022assessing}.

\section{ Main Results}

In this chapter, we prove the main results, i.e. consistency and central limit results, if the number of markers, $M_{total},$ tends to infinity. Therefore, we assume that Assumption \ref{ass:consistency:inh} holds. 

\begin{assumption}[Regularity and identifiability assumptions]
\label{ass:consistency:inh}
Let \(K\ge 2\). We assume the following conditions.

\begin{itemize}
    \item[(A1)]
    Let \(0<\kappa_q<\kappa_q'<1\) and
    \(0<r_{\mathrm{lb}}<r_{\mathrm{ub}}<\infty\). Define
    \[
    \mathcal Q
    :=
    \left\{
    q\in[\kappa_q,\kappa_q']^K:
    \sum_{k=1}^K q_k=1
    \right\}
    \]
    and
    \[
    \Theta:=\mathcal Q\times [r_{\mathrm{lb}},r_{\mathrm{ub}}].
    \]
    The true parameter satisfies \(\theta^0=(q^0,r^0)\in\Theta\).

    \item[(A2)]
    There exist constants \(0<\kappa_p<\kappa_p'<1\) such that $p_{k,m}\in[\kappa_p,\kappa_p']$ for all \(k\in\{1,\ldots,K\}\) and all markers \(m\).

    \item[(A3)]
    There exist constants \(L_Q\in\mathbb N\), \(\eta_Q>0\), and
    \(\sigma_Q>0\) such that
    \[
    \liminf_{M\to\infty}
    \frac{1}{M}
    \#\mathcal I_M
    \geq \eta_Q,
    \]
    where \(\mathcal I_M\subseteq\{1,\ldots,M-L_Q+1\}\) is the set of indices
    \(i\) for which there exist marker indices $m_1(i),\ldots,m_{K-1}(i)\in\{i,\ldots,i+L_Q-1\}$
    such that
    \[
    \sigma_{\min}
    \begin{pmatrix}
    1 & \cdots & 1 \\
    p_{1,m_1(i)} & \cdots & p_{K,m_1(i)} \\
    \vdots & & \vdots \\
    p_{1,m_{K-1}(i)} & \cdots & p_{K,m_{K-1}(i)}
    \end{pmatrix}
    \geq \sigma_Q.
    \]
    Here, \(\sigma_{\min}(\cdot)\) denotes the smallest singular value.

    \item[(A4)]
    Define, for \(m\geq 2\),
    \[
    \Delta_m(q)
    :=
    \sum_{k=1}^K q_k p_{k,m-1}p_{k,m}
    -
    \left(\sum_{k=1}^K q_k p_{k,m-1}\right)
    \left(\sum_{k=1}^K q_k p_{k,m}\right).
    \]
    There exist constants \(\eta_\Delta>0\) and \(\kappa_\Delta>0\) such that
    \[
    \liminf_{M\to\infty}
    \frac{1}{M}
    \#\left\{
    m\in\{2,\ldots,M\}:
    |\Delta_m(q^0)|\geq \kappa_\Delta
    \right\}
    \geq \eta_\Delta.
    \]

    \item[(A5)]
    There exist constants \(0<\kappa_d<d_{\mathrm{up}}<\infty\) such that $d_m\in[\kappa_d,d_{{up}}]$
    for all $m$.

    \item[(A6)]
    The number of chromosomes \(C\in\mathbb N\) is fixed.
\end{itemize}
\end{assumption}

Figure \ref{fig:overview_HMM} gives an overview of the results in this paper.
\begin{figure}[ht]
\centering
\captionsetup{skip=12pt}

\begin{tikzpicture}[node distance=0.8cm and 1.2cm]

\node (unique) [process] {
    \textbf{Unique MLE}
};

\node (consistency) [process, below=of unique] {
    \textbf{Consistency}
};

\node (clt) [process, below=of consistency] {
    \textbf{CLT}
};

\node (asymp) [process, below=of clt] {
    \textbf{Asymptotic Distribution}\\
    \textbf{of the Test Statistic}
};

\node (ass_unique) [process, fill=gray!20, right=of unique] {
    Assumption \ref{ass:consistency:inh}
};

\node (ass_consistency) [process, fill=gray!20, right=of consistency] {
    Assumption \ref{ass:consistency:inh}
};

\node (ass_clt) [process, fill=gray!20, right=of clt] {
    Assumption \ref{ass:consistency:inh}\\
    invertible Fisher\\
    information
};

\node (ass_asymp) [process,  fill=gray!20,right=of asymp] {
    Assumption \ref{ass:consistency:inh}
};

\draw [arrow] (unique) -- (consistency);
\draw [arrow] (consistency) -- (clt);
\draw [arrow] (clt) -- (asymp);

\draw [arrow] (ass_unique) -- (unique);
\draw [arrow] (ass_consistency) -- (consistency);
\draw [arrow] (ass_clt) -- (clt);
\draw [arrow] (ass_asymp) -- (asymp);

\end{tikzpicture}

\caption{Overview of results in this section and their requirements.}
\label{fig:overview_HMM}
\end{figure}

We also state the theory under the following conditions concerning the data. 

\begin{remark}[Assumptions]
    We assume that the data is haploid and that we only have bi-allelic markers. Additionally, we only consider one individual. However, extending the theory to more general cases is straightforward. 
\end{remark}

We start with the consistency of the MLE (Theorem \ref{cons:markers}). 
\begin{theorem}[Consistency of the MLE]\label{cons:markers}
   Let Assumption \ref{ass:consistency:inh} hold. 
   Then, for the MLE
   $$\left(\hat Q^{C,M}, \hat R^{C,M}\right) := \textup{argmax}\{(q, r) \mapsto\ell((x_{1,1},..., x_{C,M}), (q,r))\}$$
   it holds
   \begin{align*}
       \mathbb P\left(\lim_{M_{total} \to \infty} \bigg|\left(\hat Q^{C,M}, \hat R^{C,M}\right) - \left(q^0, r^0\right)\bigg| \geq \epsilon\right) = 0
   \end{align*}
   for any $\epsilon > 0.$
\end{theorem}

Based on Theorem \ref{cons:markers}, we can infer a CLT for the MLE, if the true parameters $\left(q^0, r^0\right)$ are in the interior of the parameter space.

\begin{theorem}[Central Limit Theorem for the MLE]\label{th:CLT}
   Let Assumption \ref{ass:consistency:inh} hold and let the Fisher information 
$$J_{q^0, r^0} := -\lim_{M_{total} \to \infty} \mathbb E\left( \frac{\partial^2}{\partial (q,r)^2}\ell\left((X_{1,1},..., X_{C,M}), (q,r))\big|_{(q,r) = \left(q^0, r^0\right)}\right)\right)$$
be positive definite. Let the true parameters be in the interior of the parameter space. Then, it holds
    \begin{align*}
        \sqrt{M_{total}}\left((\hat Q^{C,M}, \hat R^{C,M}) - (q^0, r^0)\right) \xRightarrow[]{M_{total} \to \infty} \mathcal N\left(0, J_{q^0, r^0}^{-1}\right).
    \end{align*}
\end{theorem}

\begin{remark}
    \cite{douc2005non} has proven the invertibility of the Fisher Information under certain constraints for stationary HMMs, i.e. in our case, if the allele frequencies and the distance between the markers are identical for every marker. This also means that the markers are either all on different chromosomes or on the same chromosome.  {Since there is no general theory about the invertibility of the Fisher Information yet, we as \cite{dean2014parameter}, just assume that it is invertible. For the application of the test, this is not relevant. }
\end{remark}

 \textcolor{black}{Under the usual constraints, according to \cite{wilks1938large}, it would hold
    \begin{align*}
\Lambda \xRightarrow[]{M_{total} \to \infty}\chi^2(1).        
    \end{align*}
However, the application of this result here is not possible as under $H_0$ the parameter $r$ is on the boundary of the parameter space. Hence, we use Theorem 3 by \cite{self1987asymptotic}, which can be used to calculate the asymptotic distribution of the test statistic $\Lambda,$ if the true parameter is on the boundary of the parameter space. According to \cite{self1987asymptotic}, it also holds 
    \begin{align*}
\Lambda \xRightarrow[]{M_{total} \to \infty} \frac{1}{2}\chi^2(0) + \frac{1}{2}\chi^2(1)        
    \end{align*}
    under the null hypothesis of \eqref{def:st} and some additional assumptions.  
Since it holds $\chi^2(0) = \delta_0$ and 
$$\mathbb P\left(\frac{1}{2}\chi^2(0) + \frac{1}{2}\chi^2(1) > c\right) = \frac{1}{2}\mathbb P(\chi^2(1) > c),$$
we infer that the critical value is $\chi^2_{1-2\alpha}.$ In \cite{wilks1938large} classical theory, which holds if the parameter is not on the boundary of the parameter space, the critical value would be $\chi^2_{1-\alpha},$ i.e. more conservative.}

\section{Application to Data}

 In this section, we first evaluate the performance of the test from Definition \ref{def:test} by using simulated data. Afterwards, we apply the test to data from \cite{10002015global} and compare the uncertainty of the MLEs in the Linkage Model to the ones in the Admixture Model.

\subsection{Simulated Data}

To evaluate the statistical test in terms of Type 1 and Type 2  errors, I simulated data under both the Linkage Model and the Admixture Model for different values of $M$, $r$, and the marker distances $(d_{c,m})_{c = 1, \dots, C,\ m = 1, \dots, M_c - 1}$. Specifically, I simulated $M = 50, 100, 200, 500, 1000, 10000$ markers on a single chromosome with $d_{1,m} = d_{1,m+1} =: d$ for all $m \in \{1,..., M-1\}$, where $d \in \{0.1, 1,2, 10\}$ and $r \in \{0.1, 1, 10, \infty\}$. We evaluate the performance of the test for both, diploid and haploid individuals and for $K = 2, 5$. The significance level was set to $0.05$.

I computed the Type 1 and Type 2  error rates by repeating the experiment 100 times for each possible combination of $r$ and $d$. The results are shown in Figure~\ref{fig:evaluation}.

\begin{figure}[H]
    \centering
    \begin{minipage}{0.45\textwidth}
        \centering
        \textbf{(A)}\\
        \includegraphics[width=\linewidth]{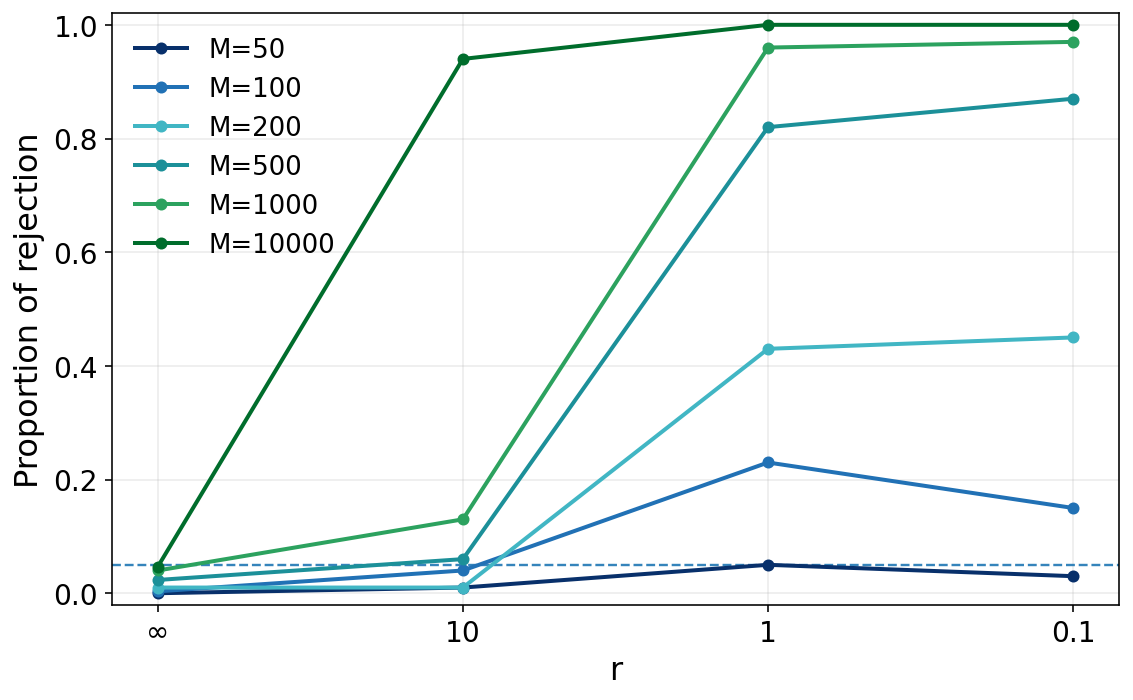}
    \end{minipage}
    \hfill
    \begin{minipage}{0.45\textwidth}
        \centering
        \textbf{(B)}\\
        \includegraphics[width=\linewidth]{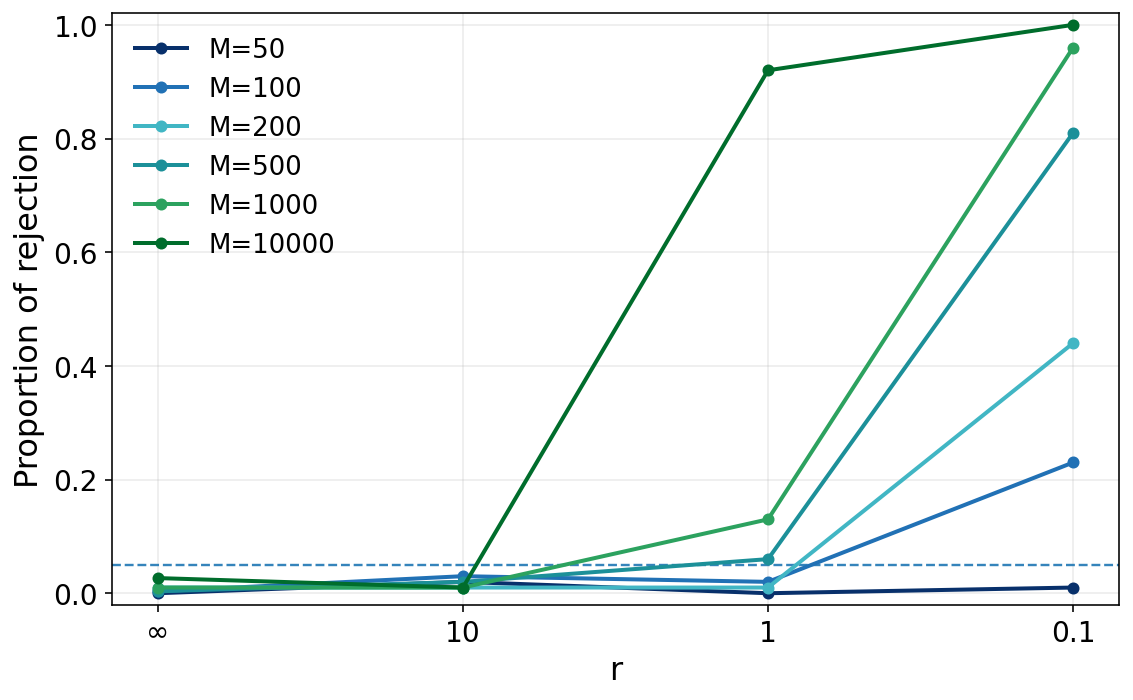}
    \end{minipage}\\[0.8em]

    \begin{minipage}{0.45\textwidth}
        \centering
        \textbf{(C)}\\
        \includegraphics[width=\linewidth]{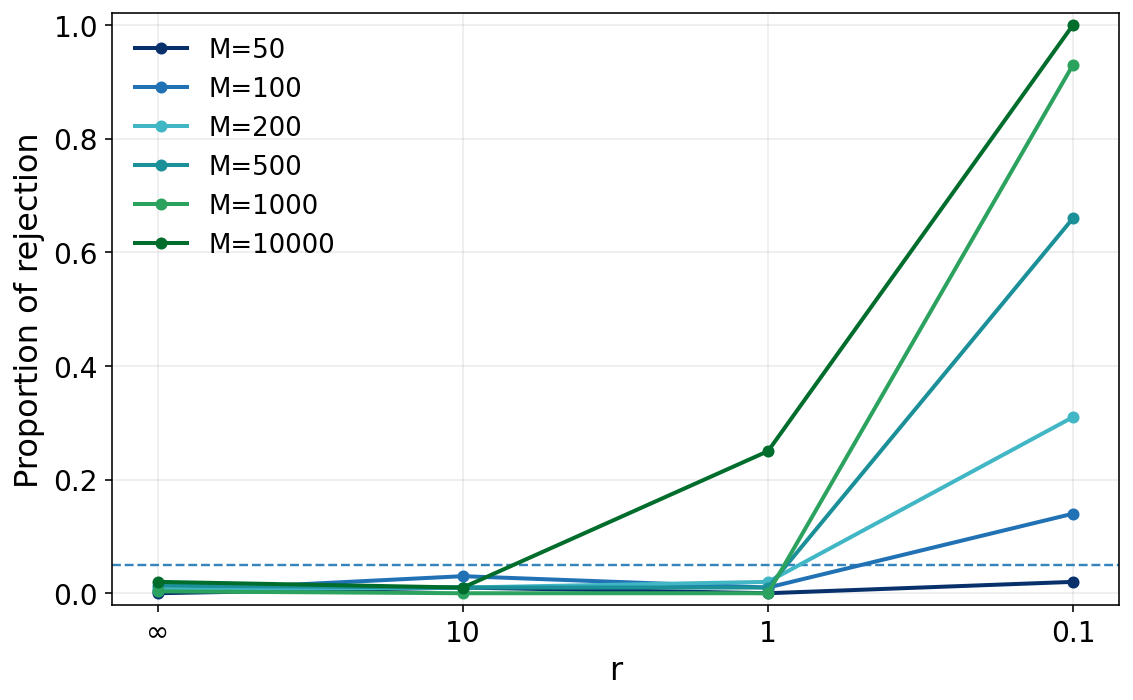}
    \end{minipage}
    \hfill
    \begin{minipage}{0.45\textwidth}
        \centering
        \textbf{(D)}\\
        \includegraphics[width=\linewidth]{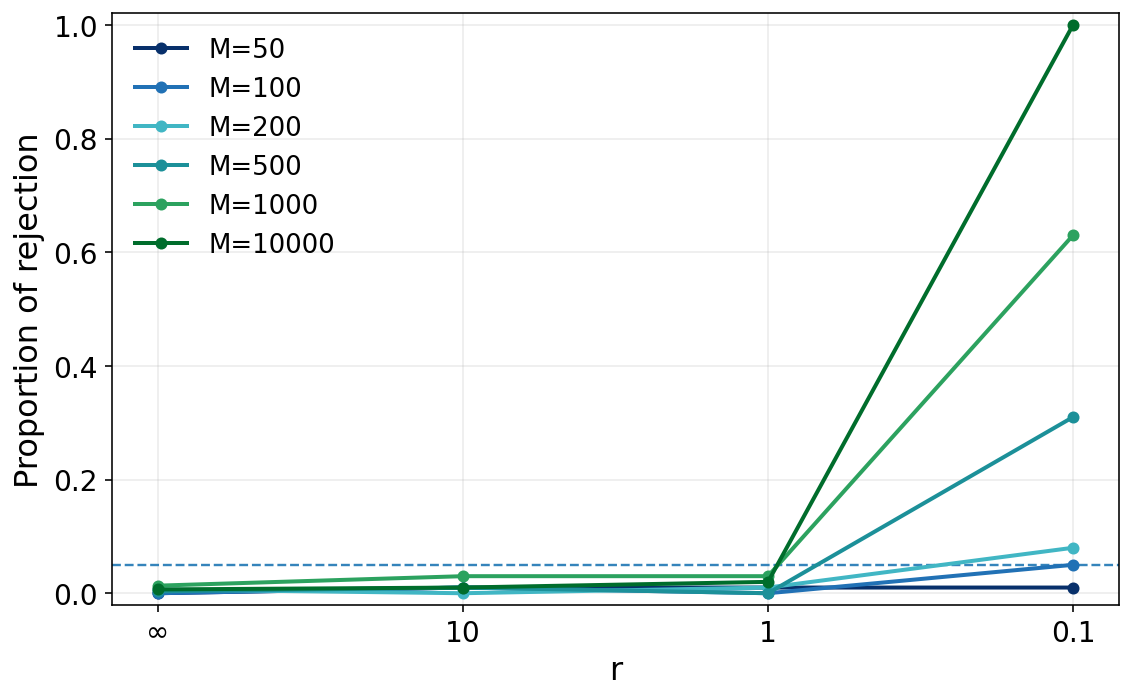}
    \end{minipage}

    \caption{Evaluation of the statistical test by using simulated data for different values of $r$ and $d$ for $K = 2$. The true ancestry was $q^0 = (0.2, 0.8).$ (A) $d_m = 0.1$, (B) $d_m = 1$, (C) $d_m = 2$, (D) $d_m = 10.$}
    \label{fig:evaluation}
\end{figure}
 {
Figure~\ref{fig:evaluation} shows that the statistical test is a level-$\alpha$ test, even for $M = 50$. Moreover, for small values of $r^0 d$, the power of the test is high and, as expected, increases with $M$. However, for large values of $r^0 d$, the power remains low even when $M = 10000$. The reason is that $\exp(-r^0d) \approx 10^{-5} \approx 0$ when $r^0d = 10$. In that case, we have
\begin{align*}
    \mathbb P(X_m = x)
    &= \sum_{k=1}^K \mathbb P(X_m = x \mid Z_m = k)\mathbb P(Z_m = k) \\
    &\approx \sum_{k=1}^K p_{k,m} q_k = \langle p_{\cdot,m}, q_{\cdot} \rangle.
\end{align*}
In particular, in this setting the test is unable to distinguish between the two models.}

 {Figure \ref{fig:evaluation_diploid} shows the power of the statistical test for the same cases as in Figure \ref{fig:evaluation}, but for diploid individuals. We see that the power is slightly higher than in the haploid case.  }
\begin{figure}[H]
    \begin{minipage}{0.45\textwidth}
        \centering
        \textbf{(A)}\\
        \includegraphics[width=\linewidth]{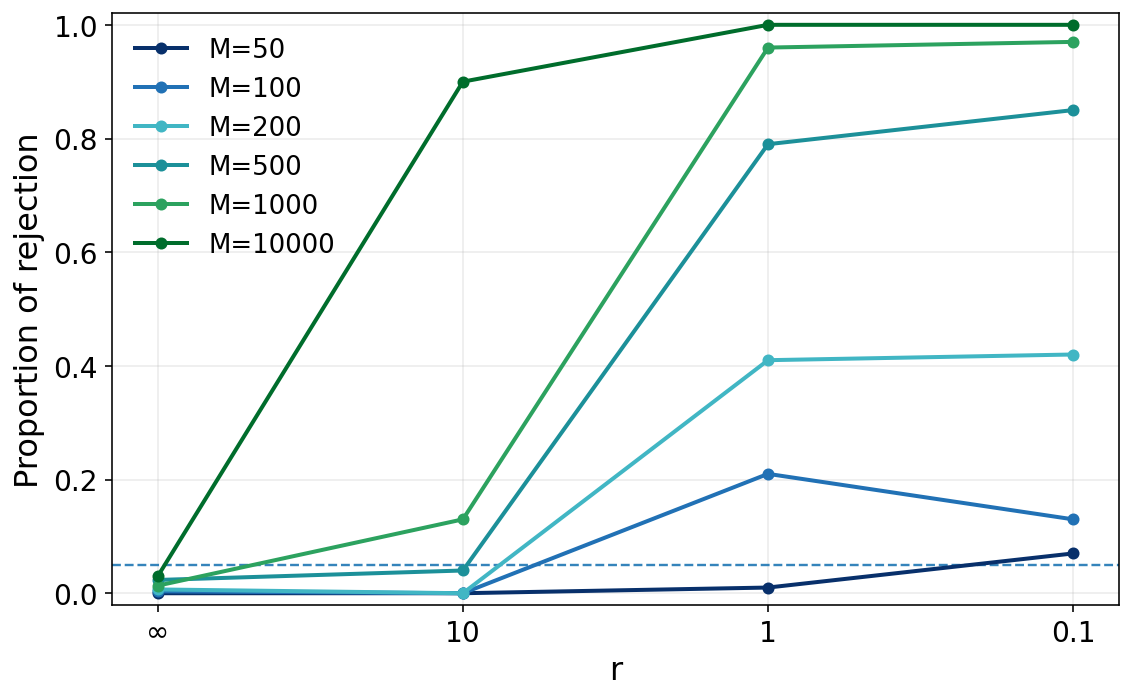}
    \end{minipage}
    \hfill
    \begin{minipage}{0.45\textwidth}
        \centering
        \textbf{(B)}\\
        \includegraphics[width=\linewidth]{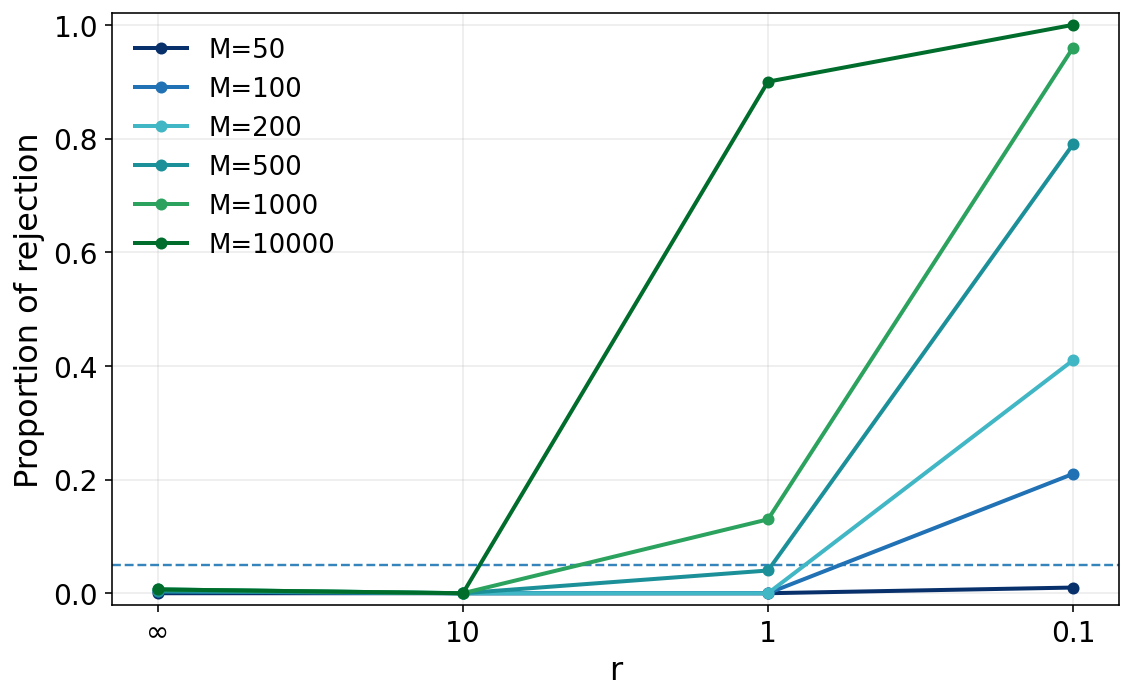}
    \end{minipage}\\[0.8em]

    \begin{minipage}{0.45\textwidth}
        \centering
        \textbf{(C)}\\
        \includegraphics[width=\linewidth]{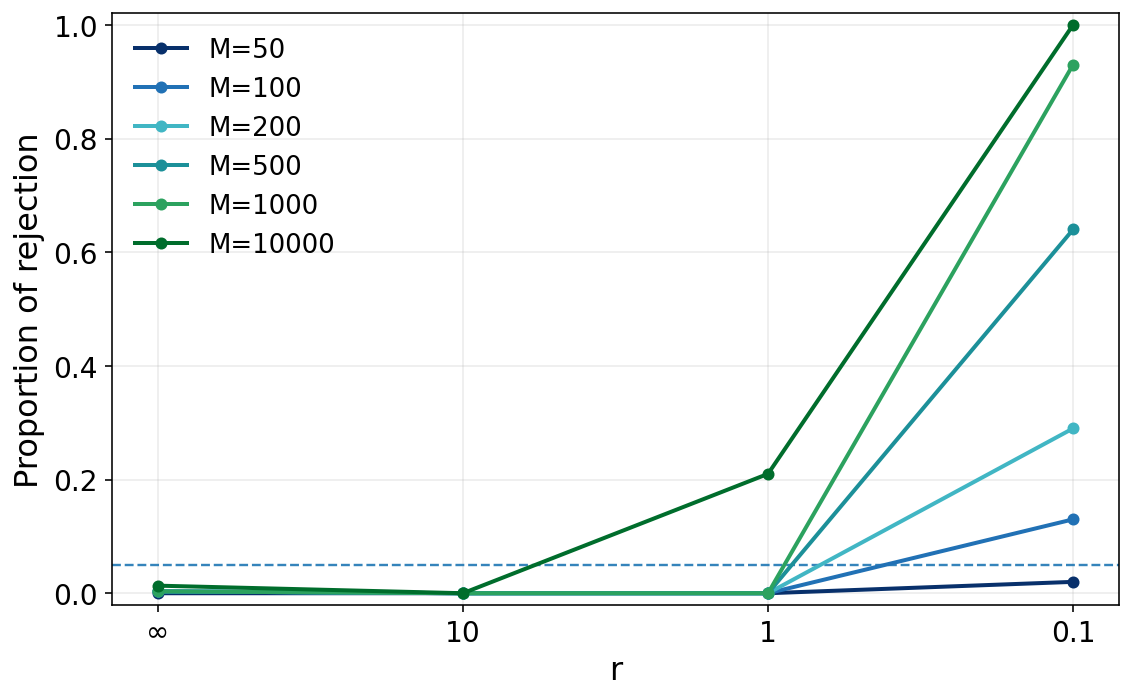}
    \end{minipage}
    \hfill
    \begin{minipage}{0.45\textwidth}
        \centering
        \textbf{(D)}\\
        \includegraphics[width=\linewidth]{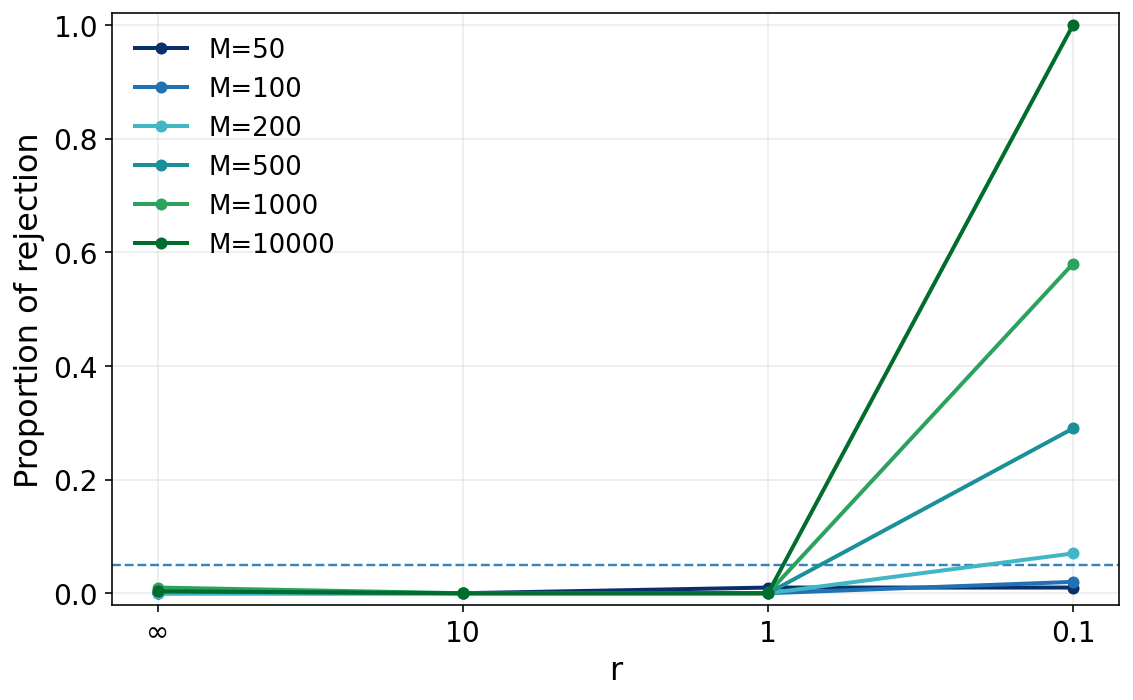}
    \end{minipage}

    \caption{Evaluation of the statistical test by using simulated data for different values of $r$ and $d$ for $K = 2$. We simulated diploid individuals. The true ancestry was $q^0 = (0.2, 0.8).$  (A) $d_m = 0.1$, (B) $d_m = 1$, (C) $d_m = 2$, (D) $d_m = 10.$}
    \label{fig:evaluation_diploid}
\end{figure}

 {Figure~\ref{fig:evaluation_K5} shows the performance of the statistical test for $K = 5$ in the haploid setting. Again, the results are comparable to those obtained for $K = 2$ in the haploid case. This suggests that the value of $K$ does not have a substantial impact on the performance of the statistical test. }

\begin{figure}[H]
    \centering
    \begin{minipage}{0.45\textwidth}
        \centering
        \textbf{(A)}\\
        \includegraphics[width=\linewidth]{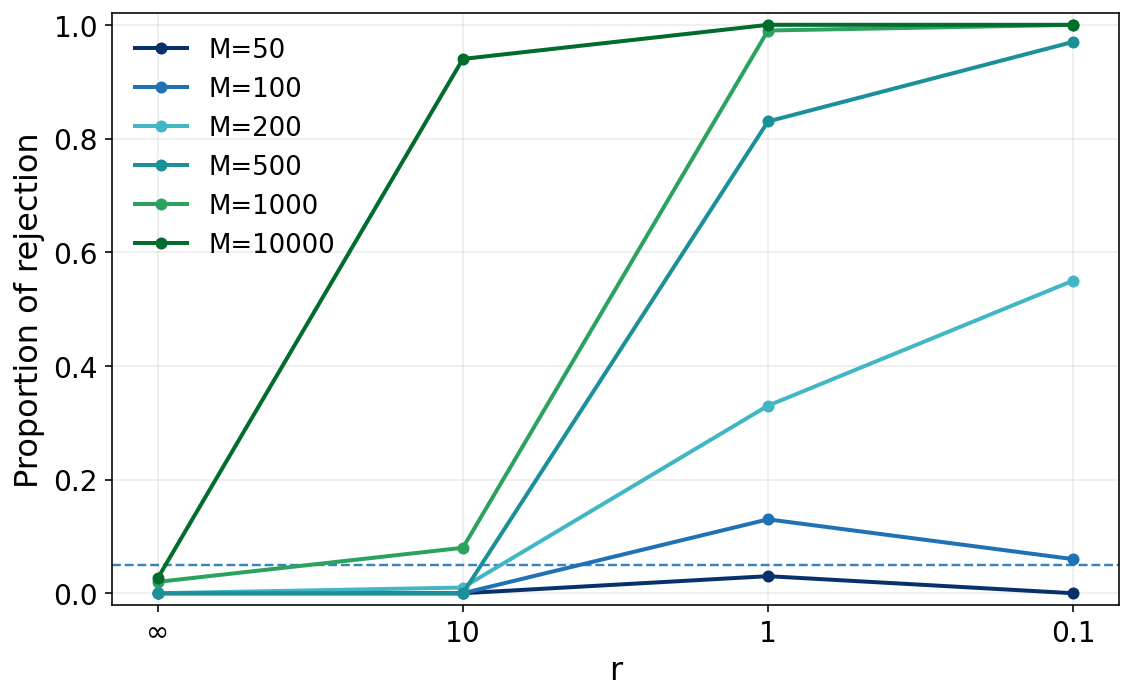}
    \end{minipage}
    \hfill
    \begin{minipage}{0.45\textwidth}
        \centering
        \textbf{(B)}\\
        \includegraphics[width=\linewidth]{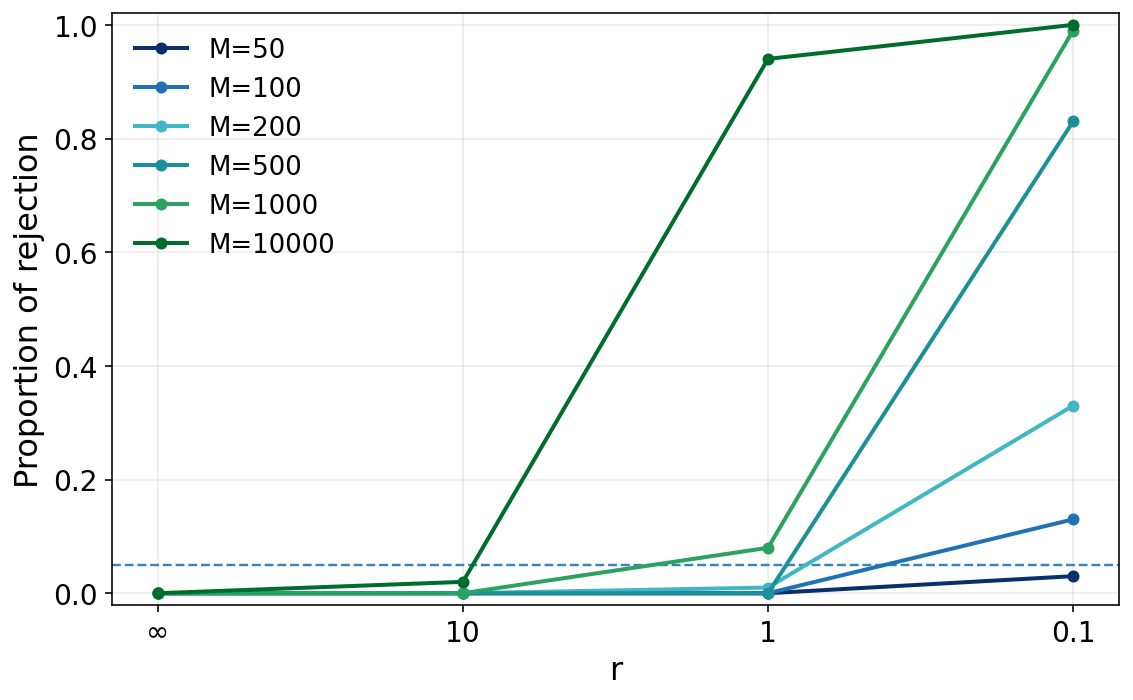}
    \end{minipage}\\[0.8em]

    \begin{minipage}{0.45\textwidth}
        \centering
        \textbf{(C)}\\
        \includegraphics[width=\linewidth]{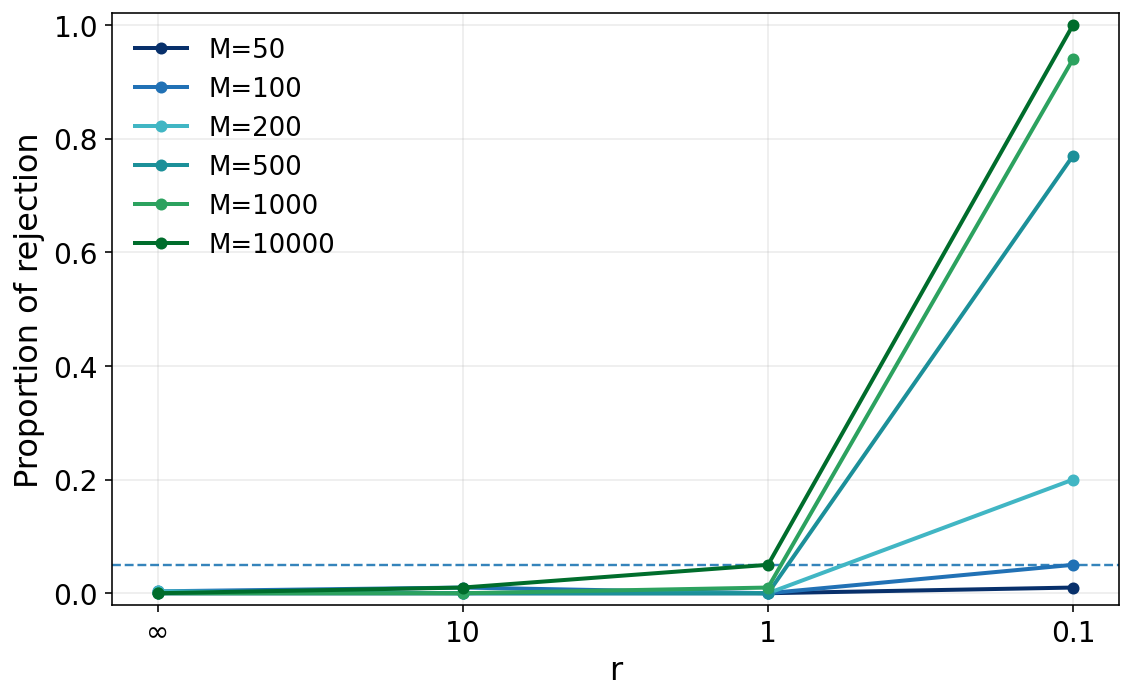}
    \end{minipage}
    \hfill
    \begin{minipage}{0.45\textwidth}
        \centering
        \textbf{(D)}\\
        \includegraphics[width=\linewidth]{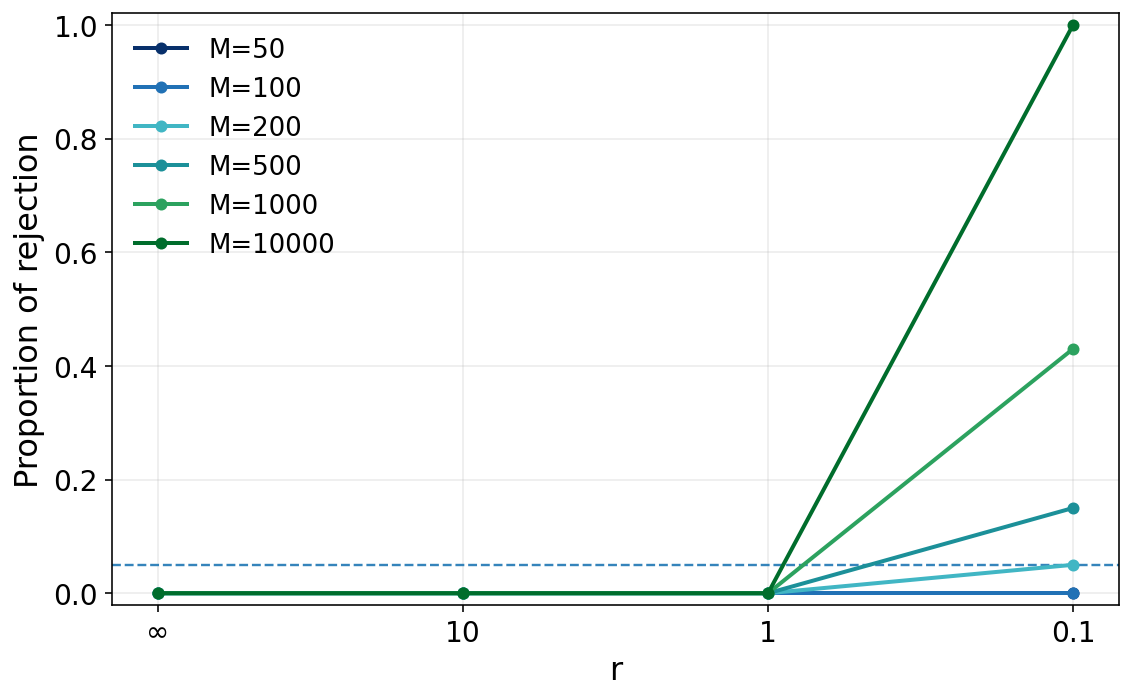}
    \end{minipage}

    \caption{Evaluation of the statistical test by using simulated data for different values of $r$ and $d$ for $K = 5$. The true ancestry was $q^0 = (0.2, 0.05, 0.25, 0.4, 0.1).$  (A) $d_m = 0.1$, (B) $d_m = 1$, (C) $d_m = 2$, (D) $d_m = 10.$}
    \label{fig:evaluation_K5}
\end{figure}

 {The evaluation of the statistical test for $K = 5$ and for diploid individuals are shown in Figure \ref{fig:evaluation_K5_diploid} in the Appendix.}

 {We also used msprime \cite{baumdicker2022efficient} to simulate data according to an island model with $K = 5$ islands. We chose the 50, 100 or 1000 markers with the hightest allele frequency difference between the populations, i.e. with the hightest
$$\max\{p_k: k =1,..., 5\} - \min\{p_k: k=1,..., 5\}.$$
With the resulting allele frequencies and distances between the markers, we then simulate the data according to the Linkage or the Admixture Model, respectively. The results are shown in Figure \ref{fig:msprime1} and Figure \ref{fig:msprime2}. Here, the power is much higher than in the previous simulations. The reason for this is that the values of $d$ are much smaller than in the simulation directly according to the model, which is shown in Figure \ref{fig:d1} in the appendix \ref{sec:msprime}. The migration rates were chosen to achieve $F_{ST} = 0.05$. Here, $F_{ST}$ is a property of the model  measuring genetic differentiation between
subpopulations relative to the total population \citep{holsinger2009genetics}. More precisely, let $\bar p := 1/K \sum_{k=1}^K p_k$. Then, it holds
$$F_{ST} = \frac{\frac 1 K \sum_{k=1}^K (p_k-\bar p)^2}{\bar p(1-\bar p)}.$$ The appendix \ref{sec:msprime}  contains a more detailed description of the simulation with msprime and for the simulations according to the model.}

\begin{figure}[H]
    \centering
    \begin{minipage}[t]{0.48\textwidth}
        \centering
        \includegraphics[width=\linewidth]{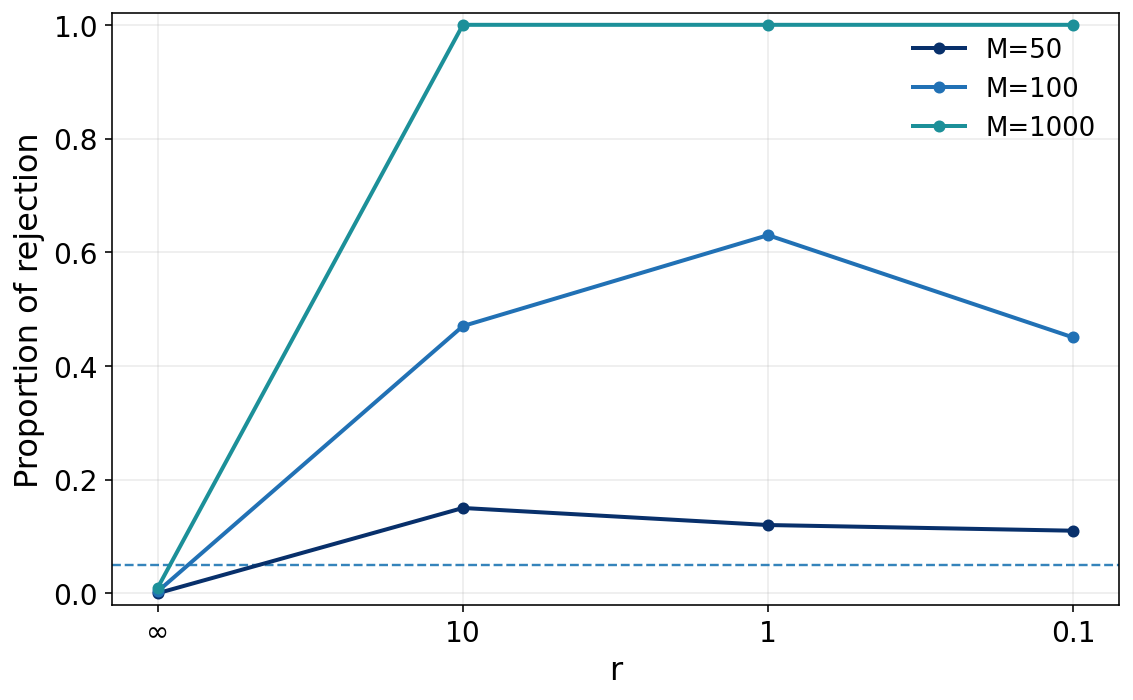}
        \caption{Evaluation of the statistical test by using simulated data for different values of $r$ and $d$ for $K = 5$,\textcolor{black}{$N_e = 1000$.} }
        \label{fig:msprime1}
    \end{minipage}
    \hfill
    \begin{minipage}[t]{0.48\textwidth}
        \centering
        \includegraphics[width=\linewidth]{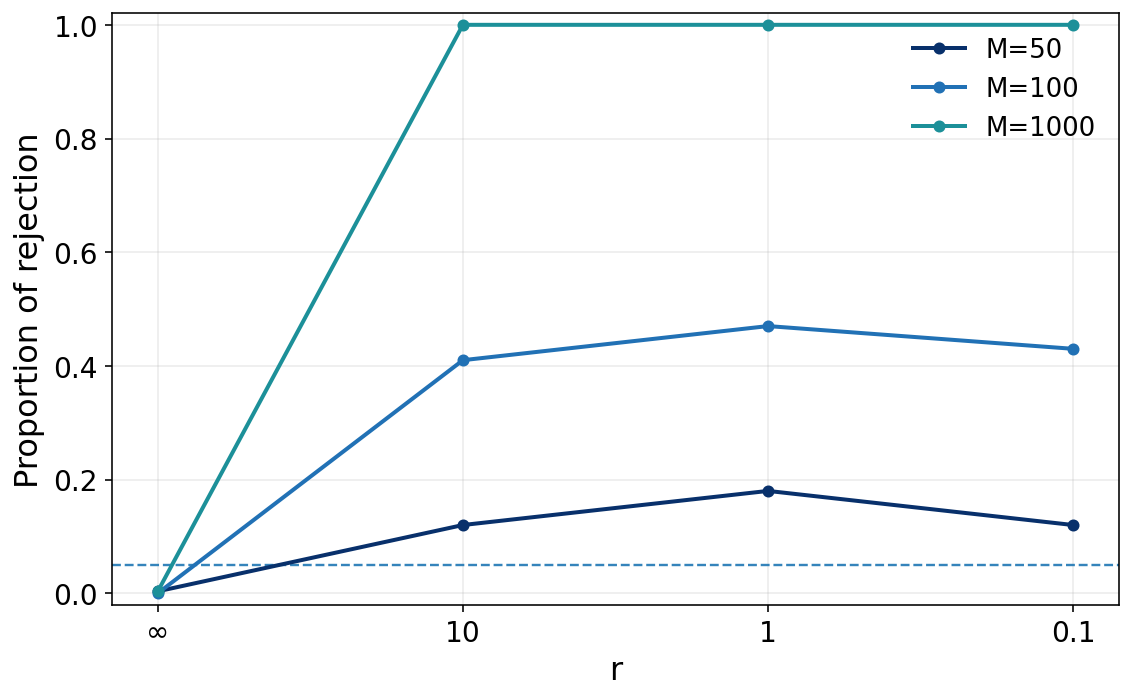}
        \caption{Evaluation of the statistical test by using simulated data for different values of $r$ and $d$ for $K = 5$,\textcolor{black}{$N_e = 5000$}.}
        \label{fig:msprime2}
    \end{minipage}
\end{figure}

\subsection{Real Data}

I also applied the statistical test to the data from \cite{10002015global} with $M = 350$ loci.  For the analysis, individuals from five populations were included: Africa (AFR), Europe (EUR), South Asia (SAS), East Asia (EAS), and Admixed Americans (AMR), i.e., $K = 5$.   {To select the loci, I first extracted every 2.000 SNP and then all bi-allelic loci, where both loci have at least $5\%$ frequency in every population. This then leads to 3340 markers. I uniformly chose $350$ markers out of these 3340 markers and apply the test to these 350 markers. The distances between the markers are shown in Figure \ref{fig:distance} in the appendix.} The code for this can be found on \href{https://github.com/CarolaHeinzel/LinkageModel}{GitHub}.

The results of the statistical test are shown in Figure \ref{fig:real}. In total, for 615 out of 2504 individuals (approximately 25\%), the null hypothesis - that the Admixture Model provides a better fit to the data - cannot be rejected.

\begin{figure}[h!]
    \centering
    \includegraphics[width=0.75\linewidth]{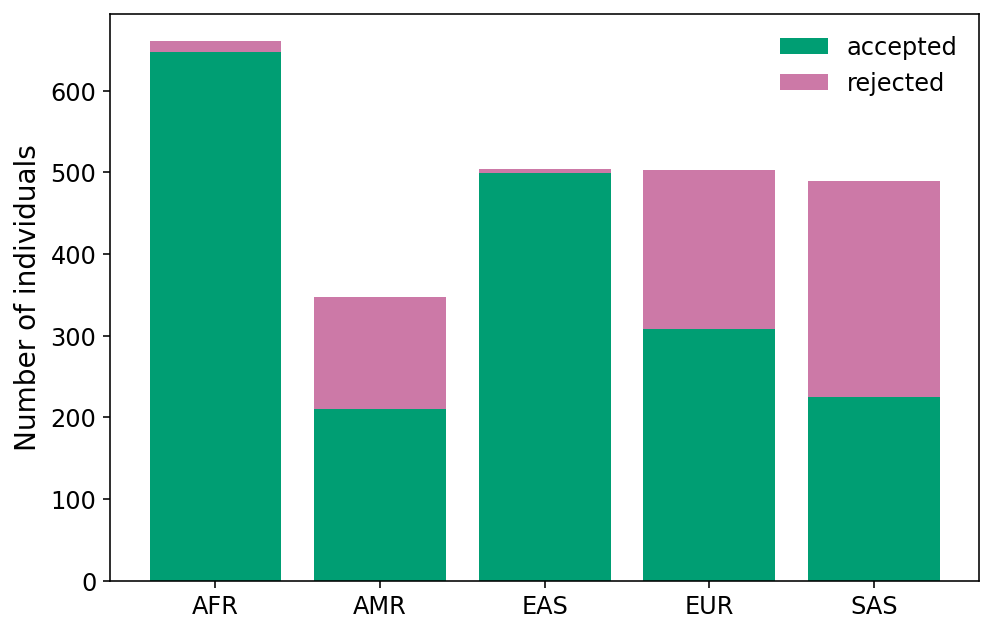}
    \caption{Results of the statistical test from Definition \ref{def:test} for the data from \cite{10002015global} with $M = 350$.}
    \label{fig:real}
\end{figure}

 {Interestingly, for EUR (38.77\%), SAS (54\%) or AMR (39.48\%), the proportion of individuals for which the null hypothesis is rejected is much higher than for individuals from EAS (0.99\%) or AFR (2.1\%) . It might be that the linkage disequibrilium between the markers in these populations is higher in EAS and AFR or that the markers are less informative in these populations to detect linkage disequibrilium. However, these are only presumption as we cannot use the test to explain these differences between the populations. Note that we also applied the test with $M = 3340$ markers. In this setting, the null hypothesis has to be rejected for all 2504 individuals.}

The results in Figure \ref{fig:real} indicates that it is not sufficient to define a general rule for choosing between the Linkage Model and the Admixture Model based solely on the marker set. Instead, a method - such as the statistical test defined in equation \eqref{def:st} - is needed, which allows model selection on an individual or population-specific basis.

\begin{remark}[Test for a whole population]
    The same method is also applicable for testing whether the Linkage Model or the Admixture Model fits better to the data of $N$ individuals. In this framework, two modeling approaches are possible: either (i) the recombination rate $r$ is assumed to vary across individuals, requiring the estimation of $N \times K$ parameters, or (ii) a single global recombination rate $r$ is assumed, reducing the number of parameters to $N \times (K - 1) + 1$. 
\end{remark}

\begin{remark}[Validity of Assumption \ref{ass:consistency:inh} for the true data]\label{rem:validity}
   {  We see in Figure \ref{fig:distance} that Assumption (A5) is met. The assumption (A2) is fulfilled for all markers and all possible combinations of the populations, while (A1) is also true due to the method how we chose the markers. However, e.g. for the $M = 55$ marker in the AIM set by \cite{kidd2014}, this assumption is violated for 12 different $p_{k,m}.$ We need (A3) to justify the invertibility of the Fisher matrix, which is also true. The MLEs for $q$ was never equal to 0 or 1 and the MLE for $r$ was never equal to 0, i.e. also the assumption about the parameter space is met. Extended information about the real data set such as the precise allele frequencies and the MLEs can be found on \href{https://github.com/CarolaHeinzel/LinkageModel}{GitHub}, respectively.} 
\end{remark}

\begin{remark}[Estimation of the Uncertainty]
    { By comparing the two models to each other, the question which model has a lower uncertainty for the MLEs arises. This question is answered in Figure \ref{fig:AM} and Figure \ref{fig:LM} in the appendix by showing the covariance matrix of the MLEs. From this figures, we conclude that the uncertainty of the MLEs for $q^0$ is comparable small for the two models, while the uncertainty of $\hat r$ is relatively large. }
\end{remark}


\section{Proofs of the Main Results}

We first prove that the MLE $\left(\hat Q^M, \hat R^M\right) := \left(\hat Q^{1,M}, \hat R^{1,M}\right)$ is asymptotically unique. Based on this, we prove consistency and central limit results for the MLE in the Linkage Model. For all proofs, we assume $C = 1$ to simplify the notation. 

To prove the consistency and the CLT, we represent the log-likelihood as a sum, i.e. it holds
\begin{align}\label{def:dm}
    &\ell((X_1,..., X_M), (q,r)) \\
    &\quad = \frac 1 M\sum_{m = 1}^M \underbrace{\log\left(\int \mathbb P_{q,r}(X_m|Z_m = z) \mathbb P_{q,r}(Z_m = dz|X_1,..., X_{m-1})\right)}_{ =: D_m^{q,r}}
\end{align}
according to \cite{van2008hidden}, Proposition 6.4, for the homogeneous case. This can easily be extended to the inhomogeneous case.

\subsection{Unique MLE}

Proving the uniqueness of a MLE for HMM has been already considered extensively for a finite state space, see e.g. \citep{finesso1990consistent, gilbert1959identifiability, blackwell1957identifiability, petrie1969probabilistic}.
In this section, we aim to prove the asymptotic uniqueness of the MLE in the Linkage Model, i.e. we aim to prove Theorem \ref{th:unique}. The idea is similar to \citep{cappe2005inference}, who used the Kullback-Leibler-divergence to prove identifiability. Therefore, we use that the limit
\begin{align}
    \ell((X_1,..., X_M), (q,r)) \xrightarrow[]{M \to \infty} \ell(q,r) \label{eq:convergence}
\end{align}
exists $\mathbb P$-almost surely. We will prove this in section \ref{sec:consistency}.

\begin{theorem}[Unique Maximum Likelihood Estimator] \label{th:unique}
   Let Assumption \ref{ass:consistency:inh} hold. Then, it holds
\begin{align}
    \ell(q^0, r^0) > \ell(q, r) \Leftrightarrow  (q, r) \neq \left(q^0, r^0\right).\label{eq:max}
\end{align}
\end{theorem}

To prove Theorem~\ref{th:unique}, we first prove that the parameters are
identifiable from a finite number of observations. Based on this
identifiability result and the limiting Kullback--Leibler contrast, we then
conclude that the limiting log-likelihood has a unique maximizer at the true
parameter \((q^0,r^0)\).

We start by proving identifiability.

\begin{lemma}[Identifiability]
\label{lem:id}
Let \(K\geq 2\) and let Assumption~\ref{ass:consistency:inh} hold. Then, the parameters \((q,r)\) are identifiable.
\end{lemma}

\begin{proof}
Let \(\theta=(q,r)\) and \(\theta^0=(q^0,r^0)\). We prove the claim by contradiction, i.e. assume that the observation laws under \(\theta\) and \(\theta^0\) coincide. We show that \(\theta=\theta^0\). For every marker \(m\), the marginal distribution of \(X_m\) is determined by $\mathbb P_\theta(X_m=1)=\langle q_{\cdot}, p_{\cdot,m}\rangle.$
Since the observation laws under \(\theta\) and \(\theta^0\) coincide, their one-dimensional marginals coincide. Hence, for every \(m\), it holds $\langle q_{\cdot}, p_{\cdot,m}\rangle=\langle q^0_{\cdot}, p_{\cdot, m}\rangle.$
By Assumption~\ref{ass:consistency:inh}(A3), there exist indices \(m_1,\ldots,m_{K-1}\) such that
\[
\operatorname{rank}
\begin{pmatrix}
1 & \cdots & 1 \\
p_{1,m_1} & \cdots & p_{K,m_1} \\
\vdots & & \vdots \\
p_{1,m_{K-1}} & \cdots & p_{K,m_{K-1}}
\end{pmatrix}
=K.
\]
We also have
$\langle q_k-q_k^0, \mathbf 1\rangle =0.$ Therefore, it holds
\[
\begin{pmatrix}
1 & \cdots & 1 \\
p_{1,m_1} & \cdots & p_{K,m_1} \\
\vdots & & \vdots \\
p_{1,m_{K-1}} & \cdots & p_{K,m_{K-1}}
\end{pmatrix}
(q-q^0)=0.
\]
Since the matrix has full rank, it follows that \(q=q^0\). It remains to identify \(r\). For \(m\geq 2\), using the transition probabilities of the hidden chain, we obtain
\[
\mathbb E_\theta(X_{m-1}X_m)
=
e^{-d_m r}
\sum_{k=1}^K q_k p_{k,m-1}p_{k,m}
+
(1-e^{-d_m r})
\left(\sum_{k=1}^K q_k p_{k,m-1}\right)
\left(\sum_{k=1}^K q_k p_{k,m}\right).
\]
Consequently, we get
\[
\operatorname{Cov}_\theta(X_{m-1},X_m)
=
e^{-d_m r}\Delta_m(q),
\]
where
\begin{align}\label{eq:def_delta}
\Delta_m(q)
=
\sum_{k=1}^K q_k p_{k,m-1}p_{k,m}
-
\left(\sum_{k=1}^K q_k p_{k,m-1}\right)
\left(\sum_{k=1}^K q_k p_{k,m}\right).
\end{align}
Since \(q=q^0\) has already been shown and since the observation laws coincide, the two-dimensional marginal distributions of \((X_{m-1},X_m)\) also coincide. Hence, it holds
\[
e^{-d_m r}\Delta_m(q^0)
=
e^{-d_m r^0}\Delta_m(q^0)
\]
for every \(m\geq 2\). By Assumption~\ref{ass:consistency:inh}(A4), there exists at least one marker \(m\) such that \(\Delta_m(q^0)\neq 0\). Therefore, it holds $\exp({-d_m r})=\exp({-d_m r^0})$. Since \(d_m>0\), this implies \(r=r^0\). Thus \(q=q^0\) and \(r=r^0\), and therefore \(\theta=\theta^0\). This proves identifiability.
\end{proof}

Now, we aim to infer \eqref{eq:max} in Theorem Theorem \ref{th:unique}from the identifiability. Therefore, we will use the following particularity of the hidden chain.

\begin{lemma}[Invariant Measure]\label{lem:inv}
Let \(d_m>0\). For \(r>0\), \(q\) is the unique invariant measure
of the transition matrix of \(Z_m\). For \(r=0\), every probability measure on \(\{1,\ldots,K\}\) is invariant.
\end{lemma}

\begin{proof}
For fixed \(m\), let \(A_m\) denote the transition matrix of the hidden chain
from \(Z_m\) to \(Z_{m+1}\). Its entries are given by
\[
A_m(k,\ell) = A_{Z_m = k, Z_{m+1} = \ell}
=
\mathbb P(Z_{m+1}=\ell\mid Z_m=k)
 .
\]
Thus, the transition probability from \(Z_m\) to \(Z_{m+1}\) is
\(A_m(Z_m,Z_{m+1})\). Let \(\pi\) be an invariant probability measure for \(A_m\). Then
\(\pi A_m=\pi\). Using the representation of \(A_m\), we obtain
\[
\pi A_m
=
e^{-r d_{m+1}}\pi
+
(1-e^{-r d_{m+1}})\pi\mathbf 1 q^\top.
\]
Since \(\pi\) is a probability measure, \(\pi\mathbf 1=1\). Hence,
\[
\pi
=
e^{-r d_{m+1}}\pi
+
(1-e^{-r d_{m+1}})q.
\]
If \(r>0\) and \(d_{m+1}>0\), then \(e^{-r d_{m+1}}<1\). Therefore,
\[
(1-e^{-r d_{m+1}})\pi
=
(1-e^{-r d_{m+1}})q,
\]
and thus \(\pi=q\). Hence, \(q\) is the unique invariant probability measure
of \(A_m\).

It remains to consider \(r=0\). In this case, \(e^{-r d_{m+1}}=1\), and
therefore \(A_m=I\).
\end{proof}

We need that $(Z_m)_{m=1, 2, ...}$ is uniformly ergodic. Therefore, we first define uniform Ergodicity.

\begin{definition}[Uniformly Ergodic]\label{def:ue}
Let $|\cdot|_{TV}$ be the total variation norm. 
     A Markov chain with transition matrices $T_i^\theta$ is called uniformly ergodic, if it holds
    \begin{align*}
        |T^\theta_n ... T^\theta_1- \pi|_{TV} \leq c_\theta \rho_\theta^n
    \end{align*}
    for $\rho_\theta < 1, c_\theta \in \mathbb R.$
\end{definition}

For Lemma \ref{lemma:erg}, we remind ourselves on the definition of uniformly ergodic from Definition \ref{def:ue}.
\begin{lemma}[Uniformly Ergodicity in the Linkage Model]\label{lemma:erg}
    Under Assumption \ref{ass:consistency:inh}, $(Z_m)_{m=1, 2, ...}$ is uniformly ergodic.
\end{lemma}
\begin{proof}
According to \cite{saloff2007convergence}, Theorem 3.3, the convergence result holds from Lemma \ref{lem:inv}.
\end{proof}

We need an additional auxiliary result, for which we define the $\chi^2$-divergence, see e.g. \citep{sason2016f}.

\begin{definition}[\(\chi^2\)-divergence]
    For two probability measures \(P\) and \(Q\) on the finite space
\(\mathcal X^{s+1}\), with densities \(p\) and \(q\) with respect to counting
measure and \(q(x)>0\) for all \(x\in\mathcal X^{s+1}\), we define the
\(\chi^2\)-divergence of \(P\) with respect to \(Q\) by
\[
\chi^2(P,Q)
:=
\sum_{x\in\mathcal X^{s+1}}
\frac{(p(x)-q(x))^2}{q(x)}.
\]

\end{definition}
\begin{lemma}[Positive-density block separation]
\label{lem:block-separation}
Let Assumption~\ref{ass:consistency:inh} hold and let
\(\theta=(q,r)\neq\theta^0=(q^0,r^0)\). Then there exist an integer
\(s\in\mathbb N\), constants \(c_\theta>0\) and \(\eta_\theta>0\), and sets $\mathcal J_{\theta,M}\subseteq\{1,\ldots,M-s\}$
such that
\[
\liminf_{M\to\infty}
\frac{1}{M}\#\mathcal J_{\theta,M}
\geq \eta_\theta
\]
and, for all \(i\in\mathcal J_{\theta,M}\),
\[
\sum_{x_{0:s}\in\mathcal X^{s+1}}
\frac{
\left(
p_{\theta^0,i}^{(s)}(x_{0:s})
-
p_{\theta,i}^{(s)}(x_{0:s})
\right)^2
}{
p_{\theta,i}^{(s)}(x_{0:s})
}
\geq c_\theta .
\]
Here, \(p_{\theta,i}^{(s)}\) denotes the density of
\((X_i,\ldots,X_{i+s})\) under the parameter \(\theta\) with respect to
counting measure.
\end{lemma}

\begin{proof}
We distinguish two cases.

First, assume that \(q\neq q^0\). 
Since both \(q\) and \(q^0\) are probability vectors, we have $\langle q_\cdot - q^0_\cdot, \mathbf 1\rangle = 0.$ Hence, for every informative window from Assumption~\ref{ass:consistency:inh}(A3),
\[
\left\|
\begin{pmatrix}
1 & \cdots & 1 \\
p_{1,m_1(i)} & \cdots & p_{K,m_1(i)} \\
\vdots & & \vdots \\
p_{1,m_{K-1}(i)} & \cdots & p_{K,m_{K-1}(i)}
\end{pmatrix}
(q-q^0)
\right\|
\geq
\sigma_Q\|q-q^0\|.
\]
Therefore, at least one marker \(m_j(i)\) in the window satisfies
\[
\left|
\sum_{k=1}^K q_k p_{k,m_j(i)}
-
\sum_{k=1}^K q_k^0 p_{k,m_j(i)}
\right|
\geq
\frac{\sigma_Q}{\sqrt{K-1}}\|q-q^0\|.
\]
But, it also holds $\mathbb P_\theta(X_m=1)
=
\langle q_\cdot, p_{\cdot, m}\rangle.$
Thus, on a positive proportion of windows, at least one one-dimensional
marginal distribution of the observations differs under \(\theta\) and
\(\theta^0\) by an amount bounded away from zero. It additionally holds  for every \(s\ge t\)
\[
\chi^2\!\left(P_{\theta^0,i}^{(s)},P_{\theta,i}^{(s)}\right)
\ge
\chi^2\!\left(P_{\theta^0,i}^{(t)},P_{\theta,i}^{(t)}\right).
\]

Now assume that \(q=q^0\) and \(r\neq r^0\). We remind ourselves on the definition of $\Delta_m(w)$ from \eqref{eq:def_delta}.
By Assumption~\ref{ass:consistency:inh}(A4), there exist constants
\(\eta_\Delta>0\) and \(\kappa_\Delta>0\) such that
\[
\liminf_{M\to\infty}
\frac{1}{M}
\#\left\{
m\in\{2,\ldots,M\}:
|\Delta_m(q^0)|\geq \kappa_\Delta
\right\}
\geq \eta_\Delta.
\]
For adjacent markers, using the transition structure of the hidden chain,
\[
\operatorname{Cov}_\theta(X_{m-1},X_m)
=
e^{-d_m r}\Delta_m(q^0),
\]
whereas
\[
\operatorname{Cov}_{\theta^0}(X_{m-1},X_m)
=
e^{-d_m r^0}\Delta_m(q^0).
\]
Since \(r\neq r^0\) and \(d_m\in[\kappa_d,d_{\mathrm{up}}]\), we have
\[
\inf_{d\in[\kappa_d,d_{\mathrm{up}}]}
\left|e^{-dr}-e^{-dr^0}\right|
>0.
\]
Consequently, on a positive proportion of adjacent marker pairs,
the joint distributions of \((X_{m-1},X_m)\) under \(\theta\) and
\(\theta^0\) differ by an amount bounded away from zero. This again implies a
positive lower bound on the corresponding \(\chi^2\)-divergence.

The case \(q\neq q^0\) and \(r\neq r^0\) is covered by the first argument.
Choosing \(s\) large enough to contain the informative markers in the first
case and the adjacent pair in the second case gives the claim.
\end{proof}

\begin{lemma}[Separating block statistics]
\label{lemma:6}
Let \(\theta^0=(q^0,r^0)\) and let
\(\theta=(q,r)\neq\theta^0\). Assume that
Assumption~\ref{ass:consistency:inh} holds. Then there exist an integer
\(s\in\mathbb N\), sets
\[
\mathcal J_{\theta,M}\subseteq\{1,\ldots,M-s\}
\]
with
\[
\liminf_{M\to\infty}
\frac{1}{M}\#\mathcal J_{\theta,M}>0,
\]
and bounded measurable functions
\[
h_i^\theta:\mathcal X^{s+1}\to\mathbb R,
\qquad i\in\mathcal J_{\theta,M},
\]
such that
\begin{align*}
\frac{1}{\#\mathcal J_{\theta,M}}
\sum_{i\in\mathcal J_{\theta,M}}
h_i^\theta(X_i,\ldots,X_{i+s})
\xrightarrow[]{M\to\infty}_{a.s.}
1
\qquad
\mathbb P_{\theta^0}\text{-a.s.}, \\
\frac{1}{\#\mathcal J_{\theta,M}}
\sum_{i\in\mathcal J_{\theta,M}}
h_i^\theta(X_i,\ldots,X_{i+s})
\xrightarrow[]{M\to\infty}_{a.s.}
0
\qquad
\mathbb P_{\theta}\text{-a.s.}.
\end{align*}
\end{lemma}

\begin{proof}
The claim and the proof are based on Lemma 6 in \cite{douc2011consistency}. 
By Lemma~\ref{lem:block-separation}, there exist an integer
\(s\in\mathbb N\), constants \(c_\theta>0\) and \(\eta_\theta>0\), and sets
\(\mathcal J_{\theta,M}\subseteq\{1,\ldots,M-s\}\) such that
\[
\liminf_{M\to\infty}
\frac{1}{M}\#\mathcal J_{\theta,M}
\geq \eta_\theta
\]
and, for all \(i\in\mathcal J_{\theta,M}\),
\[
\chi^2\left(P_{\theta^0,i}^{(s)},P_{\theta,i}^{(s)}\right)
\geq c_\theta.
\]

For \(i\in\mathcal J_{\theta,M}\), define
\[
f_i^\theta(x_{0:s})
:=
\frac{
p_{\theta^0,i}^{(s)}(x_{0:s})
}{
p_{\theta,i}^{(s)}(x_{0:s})
}.
\]
Since \(\mathcal X^{s+1}\) is finite and the emission probabilities are
bounded away from zero and one by
Assumption~\ref{ass:consistency:inh}(A2), the densities
\(p_{\theta,i}^{(s)}\) are strictly positive. Hence \(f_i^\theta\) is well
defined. Moreover, it holds $\mathbb E_{\theta}
\left(
f_i^\theta(X_i,\ldots,X_{i+s})
\right)
=1$
and
\begin{align*}
\mathbb E_{\theta^0}
\left(
f_i^\theta(X_i,\ldots,X_{i+s})
\right)
&=
\sum_{x_{0:s}}
\frac{
\left(p_{\theta^0,i}^{(s)}(x_{0:s})\right)^2
}{
p_{\theta,i}^{(s)}(x_{0:s})
} \\
&=
1+
\chi^2\left(P_{\theta^0,i}^{(s)},P_{\theta,i}^{(s)}\right).
\end{align*}
Therefore, we get $\delta_i^\theta
:=
\mathbb E_{\theta^0}
\left(
f_i^\theta(X_i,\ldots,X_{i+s})
\right)-1
\geq c_\theta$ for all \(i\in\mathcal J_{\theta,M}\). Define
\[
h_i^\theta(x_{0:s})
:=
\frac{
f_i^\theta(x_{0:s})-1
}{
\delta_i^\theta
}.
\]
Then, we have $\mathbb E_{\theta^0}
\left(
h_i^\theta(X_i,\ldots,X_{i+s})
\right)
=1,
\mathbb E_{\theta}
\left(
h_i^\theta(X_i,\ldots,X_{i+s})
\right)
=0.$ It remains to prove uniform boundedness. Since \(\mathcal X^{s+1}\) is finite
and the emission probabilities are uniformly bounded away from zero and one,
there exists a constant \(C_s<\infty\) such that
\[
\sup_{M\geq 1}
\sup_{i\in\mathcal J_{\theta,M}}
\sup_{x_{0:s}\in\mathcal X^{s+1}}
\left|f_i^\theta(x_{0:s})\right|
\leq C_s.
\]
Together with \(\delta_i^\theta\geq c_\theta\), this implies
\[
\sup_{M\geq 1}
\sup_{i\in\mathcal J_{\theta,M}}
\sup_{x_{0:s}\in\mathcal X^{s+1}}
\left|h_i^\theta(x_{0:s})\right|
<\infty.
\]

By Lemma~\ref{lemma:erg}, the hidden chain is uniformly ergodic. Since
\(h_i^\theta(X_i,\ldots,X_{i+s})\) is uniformly bounded and depends only on a
fixed finite block of observations, a strong law of large numbers for bounded
block functions of uniformly mixing non-stationary sequences yields
\[
\frac{1}{\#\mathcal J_{\theta,M}}
\sum_{i\in\mathcal J_{\theta,M}}
\left(
h_i^\theta(X_i,\ldots,X_{i+s})
-
\mathbb E_{\theta^0}
\left(
h_i^\theta(X_i,\ldots,X_{i+s})
\right)
\right)
\xrightarrow[]{M\to\infty}_{a.s.}
0
\]
under \(\mathbb P_{\theta^0}\), and
\[
\frac{1}{\#\mathcal J_{\theta,M}}
\sum_{i\in\mathcal J_{\theta,M}}
\left(
h_i^\theta(X_i,\ldots,X_{i+s})
-
\mathbb E_{\theta}
\left(
h_i^\theta(X_i,\ldots,X_{i+s})
\right)
\right)
\xrightarrow[]{M\to\infty}_{a.s.}
0
\]
under \(\mathbb P_{\theta}\). Since the corresponding expectations are \(1\)
and \(0\), the claim follows.
\end{proof}

We use Lemma \ref{lemma:6} to prove the identifiability of the parameters. This is stated in the following theorem (Theorem 7.13 in \cite{van2008hidden}). 

\begin{proof}[Proof of Theorem~\ref{th:unique}]
The direction "$\Rightarrow$" is trivial. The other direction can be proves similarly to Theorem 7.13 in \citep{van2008hidden}. Let \(\theta=(q,r)\) and \(\theta^0=(q^0,r^0)\). We show that
\[
\theta\neq\theta^0
\implies
\ell(\theta^0)>\ell(\theta).
\]
By Lemma~\ref{lem:id}, the model is identifiable. Hence
\[
\theta\neq\theta^0
\implies
\mathbb P_{\theta}^{X}\neq \mathbb P_{\theta^0}^{X}.
\]
Moreover, by Lemma~\ref{lemma:6}, for every \(\theta\neq\theta^0\) there exists
a sequence of bounded block functions \(h_i^\theta\) such that
\[
\frac1n\sum_{i=1}^n h_i^\theta(X_i,\ldots,X_{i+s})
\xrightarrow[]{n \to \infty} 1
\qquad
\mathbb P_{\theta^0}\text{-a.s.},
\]
whereas
\[
\frac1n\sum_{i=1}^n h_i^\theta(X_i,\ldots,X_{i+s})
\xrightarrow[]{n \to \infty} 0
\qquad
\mathbb P_{\theta}\text{-a.s.}.
\]
Therefore, the observation laws under \(\theta^0\) and \(\theta\) are mutually
singular.

The limiting log-likelihood contrast is the negative limiting
Kullback-Leibler contrast. Hence it is uniquely maximized at the true
parameter. 
\end{proof}

\subsection{Consistency}\label{sec:consistency}

We prove Theorem \ref{cons:markers}, i.e. the consistency of the MLE. Therefore, we follow the ideas by \cite{van2008hidden}, which are stated in the homogeneous case. First, we note that both the transition probabilities and the emission probabilities are Lipschitz continuous in the parameters. 

\begin{lemma}\label{lem:LS}
Let $d_m \leq d_{up}$ for all $m \in \{1,..., M\}.$
   There exists a constant $L_t:=  {\max\{1, d_{up}\}} \in \mathbb R$ such that
    \begin{align*}
        |\mathbb P_{q, r}(X_m = x|Z_m = z_m) - \mathbb P_{q', r'}(X_m = x|Z_m = z_m)| &\leq  |(q,r) - (q', r')| \\
        |\mathbb P_{q, r}(Z_m = z_m|Z_{m-1} = z_{m-1}) - \mathbb P_{q', r'}(Z_m = z_m|Z_{m-1} = z_{m-1})| &\leq L_t |(q,r) - (q', r')|.
    \end{align*}
\end{lemma}
\begin{proof}
It holds 
\begin{align*}
        |\mathbb P_{q, r}(X_m = x|Z_m = z_m) - \mathbb P_{q', r'}(X_m = x|Z_m = z_m)| &=  {0}.
\end{align*}
It also holds
\begin{align*}
   & |\mathbb P_{q, r}(Z_m = z_m|Z_{m-1} = z_{m-1}) - \mathbb P_{q', r'}(Z_m = z_m|Z_{m-1} = z_{m-1})| \\
   &= |(1-e^{-rd_m})q_k - (1-e^{-r'd_m})q'_k |\\
&= \big|(1-e^{-r d_m})(q_k-q'_k) + q'_k(e^{-r' d_m}-e^{-r d_m})\big| \\
&\leq |1-e^{-r d_m}|\;|q_k-q'_k| + |q'_k|\,|e^{-r d_m}-e^{-r' d_m}|\\
&\leq |q_k-q'_k| + |e^{-r d_m}-e^{-r' d_m}|.
\end{align*}

Moreover, by the mean value theorem there exists $\xi \in [r, r']$ such that
\[
|e^{-r d_m}-e^{-r' d_m}| = d_m\,|r-r'|\,e^{-\xi d_m} \le d_m\,|r-r'|.
\]
\end{proof}

We already know that the true value is the unique maximum point of the likelihood according to Theorem \ref{th:unique}. Hence, we first prove 
\begin{align}\label{eq:0}
    \sup_{(q,r) \in \Theta} |\ell((X_1,..., X_M), (q,r)) - \ell(q,r)| \xrightarrow[]{M \to \infty} 0 
\end{align}
under Assumption \ref{ass:consistency:inh}.
Based on this, we almost immediately get the claim of Theorem \ref{cons:markers}.

Proving \eqref{eq:0} is Lemma \ref{3.9} whose proof is divided into three steps:
\begin{enumerate}
    \item[1)] Prove $\lim_{M \to \infty} \frac 1 M\sum_{i=1}^M \mathbb E\left(D^{q^0, r^0}_i\right) = \mathbb E\left(\ell\left(q^0, r^0\right)\right).$ 
    \item[2)] Prove $\ell((X_1,..., X_M), (q,r))$ converges for $M \to \infty$ a.s. We call the limit  $\ell(q,r).$ This is Lemma \ref{lemma:3.10}
    \item[3)] Prove equation \eqref{eq:0}. This is Lemma \ref{3.9}.
\end{enumerate}

Lemma \ref{lemma:3.10} and Lemma \ref{3.11} are Lemma 7.8 and Lemma 7.9 in \cite{van2008hidden}, respectively, and can easily be adapted to the inhomogeneous case.
\begin{lemma}\label{lemma:3.10}
Let Assumption \ref{ass:consistency:inh} hold and let
\begin{align*}
    D_{k, \ell}^{q,r} := \log\left(\int \mathbb P_{q,r}(X_k|Z_k = z) \mathbb P_{q,r}(Z_k = z|X_{k-1}, ..., X_{k-\ell}) dz \right)
\end{align*}

There exists a constant $\alpha$ so that  \begin{align*}
       \sup_{k \in \mathbb N} |D_{k, \ell}^{q,r} - D_{k}^{q,r}| \leq \alpha(1-\epsilon^2)^\ell
   \end{align*}
   for all $\ell, k \in \mathbb N.$
\end{lemma}

\begin{lemma}\label{3.11}
    Let Assumption \ref{ass:consistency:inh} hold. There exists a constant $\alpha$ so that  \begin{align*}
       \sup_{k \in \mathbb N} |D_{k}^{q,r} - D_{k}^{q', r'}| \leq \alpha |(q, r)-(q', r')|_2
   \end{align*}
   for all $(q,r), (q', r') \in \Theta.$
\end{lemma}

Lemma \ref{3.9} is stated for the homogeneous case as Proposition 7.5 in \cite{van2008hidden}.
\begin{lemma}\label{3.9}
    Let Assumption \ref{ass:consistency:inh} hold. Then, equation \eqref{eq:0} holds.
\end{lemma}
\begin{proof}
For simplicity, we write $\theta = (q,r)$ and $\theta^0 = \left(q^0, r^0\right).$
We start chronologically:
\begin{enumerate}
\item[1)]  Since both, the transition probabilities and the emission probabilities are uniformly bounded below, we infer (according to the majorant criterion) 
    \begin{align*}
        \lim_{M \to \infty} \frac 1 M\sum_{i=1}^M \mathbb E\left(D^{q^0, r^0}_i\right) = \mathbb E\left(\ell\left(q^0, r^0\right)\right).
    \end{align*}
    \item[2)]   
    We apply Lemma 7.7 in \cite{van2008hidden} to $D_k^\theta - \mathbb E_{\theta^0}(D^\theta_k)$, for which we have to check whether there exists $\rho \in (0,1), C \in \mathbb R$ so that
    $$\mathbb E_{\theta^0}\left(D_k^\theta - \mathbb E_{\theta^0}(D^\theta_k)|X_1,..., X_\ell\right) \leq C \rho^{k-\ell}$$ holds. Therefore, we use Lemma \ref{lemma:3.10}.
    \item[3)] The details can be found in \cite{van2008hidden}. Let us mention that we first prove that $\ell(q,r)$ is also Lipschitz continuous by using Lemma \ref{lem:LS}. Then, the claim is a direct consequence of Lemma \ref{3.11} and the compactness of $\Theta$. 
\end{enumerate}

\end{proof}

\begin{proof}[Proof of Theorem \ref{cons:markers}]
We first show that if the likelihood has a unique maximum in $(q^0, r^0)$, the MLE tends to this. Therefore, we calculate $0 \leq \ell(q,r) - \ell(\hat Q^M, \hat R^M) \xrightarrow[]{M \to \infty} 0$ according to Lemma \ref{3.9}. Additionally, we know that the MLE is unique, which leads to the claim. The details are described by \cite{van2008hidden}, Theorem 7.6.
\end{proof}

\subsection{Central Limit Result}

To prove the asymptotic normality of the MLE (Theorem \ref{th:CLT}), we first prove  that the first derivative of $\ell((X_1,..., X_M), (q^0, r^0))$ is asymptotically normally distributed. Based on this, we conclude the claim with similar techniques as in \cite{hoadley1971}.

\begin{proposition}\label{th:normal}
    Let $J_{q^0, r^0} \succ 0$ and let Assumption \ref{ass:consistency:inh} hold. Then, it holds
    \begin{align*}
        \sqrt{M} \nabla \ell((X_1,...,X_M),(q^0, r^0))\xRightarrow[]{M \to \infty} \mathcal N\left(0, J^{-1}_{q^0, r^0}\right).
    \end{align*}
\end{proposition}
\begin{proof}
The proof idea is the same as the one of Proposition 5 in \citep{douc2001asymptotics}. We first define
$$N_k^{(q^0, r^0)} := \nabla D^{(q^0, r^0)}_k - \mathbb E\left(\nabla D^{(q^0, r^0)}_k\right).$$
Let $\mathcal F_n$ be the sigma field that is generated by $\nabla D^{(q^0, r^0)}_k, k = 1,..., n$. We notice that
$$\left(M^{(q^0, r^0)}_n\right)_{n \in \mathbb N} :=\left(\sum_{k=1}^n  N_k^{(q^0, r^0)} - \mathbb E\left(N_k^{(q^0, r^0)} |\mathcal F_{n-1}\right)\right)_{n \in \mathbb N}$$ is a martingale. 

Now, we check the constraints of Theorem 3.2 in \cite{hall2014martingale} to prove that $M^{(q^0, r^0)}_n/\sigma_n\xRightarrow[]{n \to \infty} \mathcal N(0, 1)$ for $$\sigma_n := \sum_{k=1}^n \left(N_k^{(q^0, r^0)}\right)^2.$$ Here, the notation omits that $\sigma_n$ depends on $q^0, r^0.$

  All three constraints (i.e. equations 3.18, 3.19 and 3.20 in \cite{hall2014martingale}), i.e.
    \begin{align*}
       & \max_k |\nabla D^{(q^0, r^0)}_k - \mathbb E(\nabla D^{(q^0, r^0)}_k)|/\sigma_n \xrightarrow[]{n \to \infty}_p 0,\\
        &\sum_{k=1}^n \left(\nabla D^{(q^0, r^0)}_k - \mathbb E\left(\nabla D^{(q^0, r^0)}_k\right)\right)^2/\sigma_n^2 \xrightarrow[]{n \to \infty}_p \eta^2 = 1, \\
        &\sup_{n}\mathbb E\left(\max_k \frac{\left(\nabla D^{(q^0, r^0)}_k - \mathbb E\left(\nabla D^{(q^0, r^0)}_k\right)\right)^2}{\sigma_n^2}\right) \leq c
    \end{align*}
    for $c \in \mathbb R,$
    follow directly by the uniformly boundedness of the transition and the emission probabilities of the hidden Markov chain. By using $\sigma^{(q^0, r^0)}_n/n$ convergences to the invertible matrix $J_{q^0, r^0}$, we can directly infer the claim.

\end{proof}

To prove Theorem \ref{th:CLT}, we need to change limes and integral at some point, for which we need to prove the boundedness of the second derivative of $D_m.$
\begin{lemma}[Uniform boundedness of second derivatives of $D_m$]\label{lemma:bounded}
 {
Let Assumption \ref{ass:consistency:inh} hold. There exists a constant $C<\infty$ such that
\[
\sup_{m\ge 1}\ \sup_{(q,r)\in\Theta}
\left|\frac{\partial^2 D_m(q,r)}{\partial \theta_a\,\partial \theta_b}\right|
\le C
\]
for every combination $\theta_a, \theta_b \in (q_1,..., q_{k-1}, r).$
}
\end{lemma}
\begin{proof}
Recall the definition of \(D_m^{q,r}\) from \eqref{def:dm}. Hence, to prove
the claim, we need to ensure that
\begin{itemize}
    \item[(i)] there exists a constant \(c>0\) such that
    \[
    \inf_m \int \mathbb P_{q,r}(X_m\mid Z_m=z)
    \mathbb P_{q,r}(Z_m=dz\mid X_1,\ldots,X_{m-1}) > c.
    \]
    \item[(ii)] the numerator of
    \[
    \frac{\partial^2 D_m^{q,r}}{\partial\theta_a\,\partial\theta_b}
    \]
    is uniformly bounded.
\end{itemize}

We start by proving (i). Let
\[
b_{m,k}:=\mathbb P(X_m=x\mid Z_m=k),
\qquad
\gamma_{m,k}:=\mathbb P(Z_m=k\mid X_{1:m-1}).
\]

Moreover, define $f_m:=\langle b_{m,\cdot},\gamma_{m,\cdot}\rangle,
$ which leads to $
D_m^{q,r}:=\log f_m.$ We first
prove that \(0<f_m<1\) for all \(m\). Set $b_-:=\min\{\kappa_p,\ 1-\kappa_p'\}>0,
b_+:=\max\{\kappa_p',\ 1-\kappa_p\}<1.$
Since \(x\in\{0,1\}\) and \(\kappa_p\le p_{k,m}\le \kappa_p'\), we have
\(b_{m,k}\in[b_-,b_+]\) for all \(m,k\). Since \(\gamma_{m,\cdot}\) is a probability
vector, \(f_m\) is a convex combination of numbers in
\([b_-,b_+]\). Hence,
\begin{equation}\label{eq:f-bounds}
b_- \le f_m \le b_+
\qquad
\forall m\ge 1.
\end{equation}

We now continue with the proof of (ii). We first need the uniform positivity
of \(\gamma_{m,k}\). Since \(\kappa_d\le d_m\) and \(r\ge r_{\mathrm{lb}}\), we have $\exp({-d_m r})
\le
\exp({-\kappa_d r_{\mathrm{lb}}})
=:\eta
<1,
1-\exp({-d_m r})
\ge
1-\eta
=:c_\lambda>0.$
For \(m\ge2\),
\[
\gamma_{m,k}
=
e^{-d_m r}\alpha_{m-1,k}
+
(1-e^{-d_m r})q_k
\ge
(1-e^{-d_m r})q_k
\ge
c_\lambda\kappa_q
=:\gamma_->0.
\]
For \(m=1\), we have \(\gamma_{1,k}=q_k\ge\kappa_q\), so the same type of lower
bound holds. Thus, uniformly in \(m\) and the parameters, it holds
\begin{equation}\label{eq:pi-lb}
\gamma_{m,k}\ge \gamma_->0
\qquad
\forall m,k.
\end{equation}

The filtering update is
\[
\alpha_{m,k}
=
\mathbb P(Z_m=k\mid X_{1:m})
=
\frac{b_{m,k}\gamma_{m,k}}{f_m}.
\]
Using
\eqref{eq:f-bounds}, \eqref{eq:pi-lb}, and \(b_{m,k}\ge b_-\), we obtain $\alpha_{m,k}
\ge
\frac{b_-\gamma_-}{b_+}
=:a_->0
\forall m,k.$ 
Let \(u:=(1/K,\ldots,1/K)\). Then \(\alpha_{m,k}\ge a_-\) implies $\alpha_m\ge \beta u,
\beta:=K a_-\in(0,1].$ Consequently, the filtering map \(F_m\), which maps
\(\alpha_{m-1}\) to \(\alpha_m\), satisfies a uniform total-variation
contraction. That is, there exists \(\rho:=1-\beta\in[0,1)\) such that, for all
probability vectors \(\alpha,\widetilde \alpha\),
\begin{equation}\label{eq:tv-contract}
\|F_m(\alpha;q,r)-F_m(\widetilde\alpha;q,r)\|_1
\le
\rho\|\alpha-\widetilde\alpha\|_1,
\end{equation}
uniformly in \(m\) and the parameters. In particular, it holds
\begin{equation}\label{eq:Jac-op}
\sup_{m\ge1}\sup_{(q,r)}\sup_\alpha
\left\|
\frac{\partial F_m}{\partial\alpha}(\alpha;q,r)
\right\|_{1\to1}
\le
\rho
<1.
\end{equation}

We now prove uniform bounds for the first and second derivatives of
\(\alpha_m\). For any free parameter
\(\theta\in\{q_1,\ldots,q_{K-1},r\}\), define
\[
J_m^{(\theta)}
:=
\frac{\partial\alpha_m}{\partial\theta}
\in\mathbb R^K.
\]
By the chain rule for \(\alpha_m=F_m(\alpha_{m-1};q,r)\), we get
\[
J_m^{(\theta)}
=
\frac{\partial F_m}{\partial\alpha}(\alpha_{m-1};q,r)
J_{m-1}^{(\theta)}
+
\frac{\partial F_m}{\partial\theta}(\alpha_{m-1};q,r).
\]
All entries of \(\partial F_m/\partial\theta\) are rational expressions in
\(b_{m,k}\), \(\gamma_{m,k}\), \(f_m\), and, for \(\theta=r\), in
\(e^{-d_m r}\) and \(-d_m e^{-d_m r}\). Their denominators are bounded away
from zero by \eqref{eq:f-bounds}--\eqref{eq:pi-lb}, and their numerators are
uniformly bounded. Hence, there exists \(B_1<\infty\) such that
\[
\sup_{m\ge1}\sup_{(q,r)}\sup_\alpha
\left\|
\frac{\partial F_m}{\partial\theta}(\alpha;q,r)
\right\|_1
\le
B_1.
\]
Together with \eqref{eq:Jac-op}, this yields
\[
\|J_m^{(\theta)}\|_1
\le
\rho\|J_{m-1}^{(\theta)}\|_1
+
B_1,
\]
and therefore
\[
\sup_{m\ge1}\|J_m^{(\theta)}\|_1<\infty
\]
by geometric summation. Similarly, for free parameters \(\theta_a,\theta_b\), define
\[
H_m^{(\theta_a,\theta_b)}
:=
\frac{\partial^2\alpha_m}
{\partial\theta_a\,\partial\theta_b}
\in\mathbb R^K.
\]
A second-order chain rule gives
\[
\begin{aligned}
H_m^{(\theta_a,\theta_b)}
={}&
\frac{\partial F_m}{\partial\alpha}
H_{m-1}^{(\theta_a,\theta_b)}
+
\frac{\partial^2 F_m}{\partial\alpha^2}
\!\left[
J_{m-1}^{(\theta_a)},
J_{m-1}^{(\theta_b)}
\right] \\
&+
\frac{\partial^2 F_m}{\partial\alpha\,\partial\theta_a}
J_{m-1}^{(\theta_b)}
+
\frac{\partial^2 F_m}{\partial\alpha\,\partial\theta_b}
J_{m-1}^{(\theta_a)}
+
\frac{\partial^2 F_m}{\partial\theta_a\,\partial\theta_b}.
\end{aligned}
\]
All second derivatives of \(F_m\) are uniformly bounded for the same reason as
above, now also using
\[
d_m^2e^{-d_m r}
\le
d_{\mathrm{up}}^2\eta,
\]
and the already established uniform bounds on \(J_{m-1}\). Hence, there exists
\(B_2<\infty\) such that
\[
\|H_m^{(\theta_a,\theta_b)}\|_1
\le
\rho\|H_{m-1}^{(\theta_a,\theta_b)}\|_1
+
B_2.
\]
Thus, we follow
\[
\sup_{m\ge1}
\|H_m^{(\theta_a,\theta_b)}\|_1
<\infty.
\]

Based on the previous steps, we can conclude the uniform boundedness of
\(\nabla^2 D_m^{q,r}\). Since
\[
\gamma_{m,\cdot}
=
e^{-d_m r}\alpha_{m-1}
+
(1-e^{-d_m r})q,
\]
the bounds on \(\alpha_{m-1}\), \(J_{m-1}\), and \(H_{m-1}\), together with the
uniform bounds on $\exp({-d_m r}),
d_m \exp({-d_m r}),
d_m^2 \exp({-d_m r}),$
imply that \(\gamma_{m,k}\) and its first and second derivatives with respect to
any free parameters are uniformly bounded in \(m\). Therefore,
\[
f_m=\langle b_{m,\cdot},\gamma_{m,\cdot}\rangle
\]
and its first and second derivatives are uniformly bounded as well. Finally, we note that
\[
\frac{\partial^2 D_m^{q,r}}
{\partial\theta_a\,\partial\theta_b}
=
\frac{
f_m f_{m,ab}
-
f_{m,a}f_{m,b}
}
{f_m^2},
\]
where $f_{m,a}
=
\langle b_{m,\cdot},\gamma_{m,a}\rangle,
f_{m,ab}
=
\langle b_{m,\cdot},\gamma_{m,ab}\rangle.$
Moreover, \(f_m^2\ge b_-^2>0\) by \eqref{eq:f-bounds}. 
\end{proof}
Finally, we are ready to prove the asymptotic normality of $\hat Q^M, \hat R^M$.

\begin{proof}[Proof of Theorem \ref{th:CLT}]
    We proceed similar to \cite{hoadley1971}. For simplicity, we write $\theta = (q,r)$. With probability $\epsilon_M \xrightarrow[]{M \to \infty} 1$, it holds
    \begin{align*}
        0 = \nabla \ell((X_1,..., X_M), \hat \theta^M).
    \end{align*}
    Hence,  we get
    \begin{align*}
        &\nabla \ell((X_1,..., X_M), \theta)|_{\theta = \hat \theta^M} - \nabla \ell((X_1,..., X_M), \theta^0) \\
        &= (\hat \theta^M - \theta^0)
        \cdot \int_0^1 \nabla^2 \ell((X_1,..., X_M), \theta^0 + (\hat \theta^M - \theta^0) \xi)d\xi.
    \end{align*}
    We define $I_M(\theta^0) := \int_0^1 \frac{1}{M} \sum_{k=1}^M \nabla^2 D_k((X_1,..., X_M), \theta^0 + (\hat \theta^M - \theta^0) \xi)d\xi.$
Consequently, 
\begin{align*}
   \sqrt{M}(\hat \theta^M - \theta^0) I_M(\theta^0) = \frac{1}{\sqrt{M}} \sum_{k=1}^M \nabla D_k((X_1,..., X_M), \theta^0).
\end{align*}
We prove $I_M(\theta^0) \xrightarrow[]{M \to \infty}_p J\left(\theta^0\right).$  It holds
\begin{align*}
    &\lim_{M \to \infty} |I_M\left(\theta^0\right) - J\left(\theta^0\right)| \\
    &= \lim_{M \to \infty} \int_0^1 \frac{1}{M}\sum_{k=1}^M \left( \nabla^2 D_k((X_1,..., X_M), \theta^0 + (\hat \theta^M - \theta^0) \xi) - \mathbb E\left(\nabla^2 D_k((X_1,..., X_M), \theta^0)\right)\right)d\xi \\
    &= 0,
\end{align*}
where we used dominated convergence.  {Therefore, we have to ensure that $\nabla^2 D_k((X_1,..., X_M), \theta^0)$ is bounded a.s., which is Lemma \ref{lemma:bounded}.} Then, with Proposition \ref{th:normal} and the assumption that the matrix $J\left(\theta^0\right)$ is invertible, we conclude the claim. 
\end{proof}

\begin{remark}[Extension to diploid data]\label{rem:missing}
 {To extend our proofs to general ploidy, especially to diploid data, we would have to deal with the uniqueness of the MLE in this case and we would also have to extend the assumption about the invertibility of the Fisher information to the diploid case. In particular, the proof ideas remain the same.}
\end{remark}

\section{Discussion}
From a biological perspective, the Linkage Model is useful because it explains genetic data in a relatively simple framework while accounting for linkage. From a mathematical point of view, it is particularly interesting as the stationary distribution of the hidden Markov chain remains the same across all markers, even though the underlying Markov chain is inhomogeneous. This specific model has already been considered in the context of Markov chains \citep{saloff2007convergence}. In this work, we use this property to prove limit results for the MLE in the Linkage Model.

Specifically, we investigate the consistency and the asymptotic normality of the MLE in the Linkage Model. In doing so, we prove the uniqueness of the MLE in this model. This is an important result, as in the Admixture Model-even in the supervised setting-the MLE is sometimes non-unique \citep{pfaffelhuber2022central, heinzel_g3, HEINZEL2025}.

The theoretical results can be applied in several ways, for example in marker selection, similar to \cite{pfaff2004information}, who proposed using CLTs for the Admixture Model. Marker selection remains a widely studied topic \citep{phillips2014,kidd2014,xavier2022,xavier2020development,pfaffelhuber2020,resutik2023,phillips2019maplex,kosoy2009ancestry}, and the presented results can contribute to assessing the quality of a marker set. Furthermore, the CLT represents the first published approach to quantify the uncertainty of the MLE in the Linkage Model.

Arguably, the most important application of the CLT is that it provides a theoretical foundation for a statistical test to compare the Linkage Model with the Admixture Model. This test helps determine which model better fits a given dataset. To my knowledge, this is the first data-based model selection method for the Linkage Model.

Of course, both the Admixture Model and the Linkage Model are simplifications of biological reality. This study only addresses the question of which of the two models fits a dataset better. It is still an open problem whether either model describes the data appropriately. Future work could include goodness-of-fit tests to evaluate whether either model is adequate at all. It would also be interesting to compare the performance of different model selection methods, such as cross-validation \citep{anderson2004model}, with the statistical test developed here.

In this study, we only considered the supervised setting, i.e., the allele frequencies are assumed to be known. However, in the unsupervised setting, determining the number of ancestral populations is a major challenge in population genetics \citep{evanno2005, wang2019parsimony, pritchard2000, raj2014, verity2016, alexander2011}. So far, none of the existing methods perform well \citep{garcia2020evaluation}. A promising approach for the Linkage Model-possibly also applicable to the Admixture Model-could be methods for choosing the order of a HMM, such as those proposed in \cite{van2008hidden, mackay2002estimating}. 

 {Our theoretical results only hold for $M \to \infty$. However, there are many attempts to increase the power of the statistical test for smaller values of $M.$ This included bootstrap with Edgeworth expansion \cite{hall1992principles} or the Bartlett correction \citep{Lawley1956, Hayakawa1977, Cordeiro1987}. However, the theoretical statements only work for i.i.d. random variables here. Extending this theory to the Linkage Model and applying it to our test problem \eqref{def:st} might increase the power of the statistical test.}

\noindent
\textbf{Data availability} \\
\noindent
The python scripts (application of the statistical test, evaluation of the statistical test and calculation of the variance of the MLEs) are available on \href{https://github.com/CarolaHeinzel/LinkageModel}{GitHub} (\url{https://github.com/CarolaHeinzel/LinkageModel}).\\

\noindent
\textbf{Funding} \\
\noindent
CSH is funded by the Deutsche Forschungsgemeinschaft (DFG, German Research Foundation) – Project-ID 499552394 – SFB Small Data. \\ 

\noindent
\textbf{Acknowledgment} \\
\noindent
CSH thanks two anonymous referees, who significantly improved the representation of the result. \\

\noindent
\textbf{Declaration of Conflicts of Interest} \\
The author declares that there are no conflicts of interest.

\begin{small}
    \bibliography{Bibliography,HMM,phasing}

@article{xavier2020development,
  title={Development and validation of the VISAGE AmpliSeq basic tool to predict appearance and ancestry from DNA},
  author={Xavier, Catarina and de la Puente, Maria and Mosquera-Miguel, Ana and Freire-Aradas, Ana and Kalamara, Vivian and Vidaki, Athina and Gross, Theresa E and Revoir, Andrew and Po{\'s}piech, Ewelina and Kartasi{\'n}ska, Ewa and others},
  journal={Forensic Science International: Genetics},
  volume={48},
  pages={102336},
  year={2020},
  publisher={Elsevier}
}

@article{kosoy2009ancestry,
  title={Ancestry informative marker sets for determining continental origin and admixture proportions in common populations in America},
  author={Kosoy, Roman and Nassir, Rami and Tian, Chao and White, Phoebe A and Butler, Lesley M and Silva, Gabriel and Kittles, Rick and Alarcon-Riquelme, Marta E and Gregersen, Peter K and Belmont, John W and others},
  journal={Human mutation},
  volume={30},
  number={1},
  pages={69--78},
  year={2009},
  publisher={Wiley Online Library}
}

@article{resutik2023,
  title={Comparative evaluation of the MAPlex, Precision ID Ancestry Panel, and VISAGE Basic Tool for biogeographical ancestry inference},
  author={Resutik, Peter and Aeschbacher, Simon and Kr{\"u}tzen, Michael and Kratzer, Adelgunde and Haas, Cordula and Phillips, Christopher and Arora, Natasha},
  journal={Forensic Science International: Genetics},
  volume={64},
  pages={102850},
  year={2023},
  publisher={Elsevier}
}

@article{phillips2019maplex,
  title={MAPlex-A massively parallel sequencing ancestry analysis multiplex for Asia-Pacific populations},
  author={Phillips, C and McNevin, D and Kidd, KK and Lagac{\'e}, R and Wootton, S and De La Puente, M and Freire-Aradas, A and Mosquera-Miguel, A and Eduardoff, M and Gross, T and others},
  journal={Forensic Science International: Genetics},
  volume={42},
  pages={213--226},
  year={2019},
  publisher={Elsevier}
}

@article{heinzel_g3,
    author = {Heinzel, Carola Sophia and Baumdicker, Franz and Pfaffelhuber, Peter},
    title = {Revealing the range of equally likely estimates in the admixture model},
    journal = {G3 Genes|Genomes|Genetics},
    pages = {jkaf142},
    year = {2025},
    month = {06},
    issn = {2160-1836},
    doi = {10.1093/g3journal/jkaf142},
    url = {https://doi.org/10.1093/g3journal/jkaf142},
    eprint = {https://academic.oup.com/g3journal/advance-article-pdf/doi/10.1093/g3journal/jkaf142/63525579/jkaf142.pdf},
}

@article{wilks1938large,
  title={The large-sample distribution of the likelihood ratio for testing composite hypotheses},
  author={Wilks, Samuel S},
  journal={The annals of mathematical statistics},
  volume={9},
  number={1},
  pages={60--62},
  year={1938},
  publisher={JSTOR}
}

@Article{pfaffelhuber2022,
  author    = {Pfaffelhuber, Peter and Sester-Huss, Elisabeth and Baumdicker, Franz and Naue, Jana and Lutz-Bonengel, Sabine and Staubach, Fabian},
  journal   = {Forensic Science International: Genetics},
  title     = {Inference of recent admixture using genotype data},
  year      = {2022},
  pages     = {102593},
  volume    = {56},
  comment   = {Bestimme die Wahrscheinlichkeit für q gegeben G},
  publisher = {Elsevier},
}

@article{pfaff2004information,
  title={Information on ancestry from genetic markers},
  author={Pfaff, Carrie Lynn and Barnholtz-Sloan, Jill and Wagner, Jennifer K and Long, Jeffrey C},
  journal={Genetic Epidemiology: The Official Publication of the International Genetic Epidemiology Society},
  volume={26},
  number={4},
  pages={305--315},
  year={2004},
  publisher={Wiley Online Library}
}

@Article{pfaffelhuber2020,
  author    = {Pfaffelhuber, Peter and Grundner-Culemann, Franziska and Lipphardt, Veronika and Baumdicker, Franz},
  journal   = {Forensic Science International: Genetics},
  title     = {How to choose sets of ancestry informative markers: A supervised feature selection approach},
  year      = {2020},
  pages     = {102259},
  volume    = {46},
  comment   = {Wie werden AIMs ausgewählt?},
  publisher = {Elsevier},
}

@Article{hoadley1971,
  author    = {Hoadley, Bruce},
  journal   = {The Annals of mathematical statistics},
  title     = {Asymptotic properties of maximum likelihood estimators for the independent not identically distributed case},
  year      = {1971},
  pages     = {1977--1991},
  publisher = {JSTOR},
}

@Article{pritchard2000,
  author    = {Pritchard, Jonathan K and Stephens, Matthew and Donnelly, Peter},
  journal   = {Genetics},
  title     = {Inference of population structure using multilocus genotype data},
  year      = {2000},
  number    = {2},
  pages     = {945--959},
  volume    = {155},
  comment   = {STRUCTURE},
  publisher = {Oxford University Press},
}

@Article{alexander2011,
  author    = {Alexander, David H and Lange, Kenneth},
  journal   = {BMC bioinformatics},
  title     = {Enhancements to the ADMIXTURE algorithm for individual ancestry estimation},
  year      = {2011},
  pages     = {1--6},
  volume    = {12},
  publisher = {Springer},
}

@Article{10002015global,
  author    = {{The 1000 Genomes Project Consortium}},
  journal   = {Nature},
  title     = {A global reference for human genetic variation},
  year      = {2015},
  number    = {7571},
  pages     = {68-74},
  volume    = {526},
  publisher = {Nature Publishing Group},
}

@Article{kidd2014,
  author    = {Kidd, Kenneth K and Speed, William C and Pakstis, Andrew J and Furtado, Manohar R and Fang, Rixun and Madbouly, Abeer and Maiers, Martin and Middha, Mridu and Friedlaender, Fran{\c{c}}oise R and Kidd, Judith R},
  journal   = {Forensic Science International: Genetics},
  title     = {Progress toward an efficient panel of SNPs for ancestry inference},
  year      = {2014},
  pages     = {23--32},
  volume    = {10},
  publisher = {Elsevier},
}

@Article{xavier2022,
  author    = {Xavier, Catarina and de la Puente, Maria and Sidstedt, Maja and Junker, Klara and Minawi, Angelika and Unterl{\"a}nder, Martina and Chantrel, Yann and Laurent, Fran{\c{c}}ois-Xavier and Delest, Anna and Hohoff, Carsten and others},
  journal   = {Forensic Science International: Genetics},
  title     = {Evaluation of the VISAGE basic tool for appearance and ancestry inference using ForenSeq{\textregistered} chemistry on the MiSeq FGx{\textregistered} system},
  year      = {2022},
  pages     = {102675},
  volume    = {58},
  publisher = {Elsevier},
}

@Article{phillips2014,
  author    = {Phillips, Christopher and Parson, W and Lundsberg, B and Santos, C and Freire-Aradas, A and Torres, M and Eduardoff, M and B{\o}rsting, C and Johansen, P and Fondevila, M and others},
  journal   = {Forensic Science International: Genetics},
  title     = {Building a forensic ancestry panel from the ground up: The EUROFORGEN Global AIM-SNP set},
  year      = {2014},
  pages     = {13--25},
  volume    = {11},
  publisher = {Elsevier},
}

@Article{raj2014,
  author    = {Raj, Anil and Stephens, Matthew and Pritchard, Jonathan K},
  journal   = {Genetics},
  title     = {fastSTRUCTURE: variational inference of population structure in large SNP data sets},
  year      = {2014},
  number    = {2},
  pages     = {573--589},
  volume    = {197},
  publisher = {Oxford University Press},
}

@Article{falush2003,
  author    = {Falush, Daniel and Stephens, Matthew and Pritchard, Jonathan K},
  journal   = {Genetics},
  title     = {Inference of population structure using multilocus genotype data: linked loci and correlated allele frequencies},
  year      = {2003},
  number    = {4},
  pages     = {1567--1587},
  volume    = {164},
  publisher = {Oxford University Press},
}

@article{evanno2005,
  title={Detecting the number of clusters of individuals using the software STRUCTURE: a simulation study},
  author={Evanno, Guillaume and Regnaut, Sebastien and Goudet, J{\'e}r{\^o}me},
  journal={Molecular ecology},
  volume={14},
  number={8},
  pages={2611--2620},
  year={2005},
  publisher={Wiley Online Library}
}

@article{verity2016,
  title={Estimating the number of subpopulations (K) in structured populations},
  author={Verity, Robert and Nichols, Richard A},
  journal={Genetics},
  volume={203},
  number={4},
  pages={1827--1839},
  year={2016},
  publisher={Oxford University Press}
}

@article{pfaffelhuber2022central,
  title={A central limit theorem concerning uncertainty in estimates of individual admixture},
  author={Pfaffelhuber, Peter and Rohde, Angelika},
  journal={Theoretical Population Biology},
  volume={148},
  pages={28--39},
  year={2022},
  publisher={Elsevier}
}

@article{douc2011consistency,
  title={Consistency of the maximum likelihood estimator for general hidden Markov models},
  author={Douc, Randal and Moulines, Eric and Olsson, Jimmy and Van Handel, Ramon},
  year={2011}
}

@book{finesso1990consistent,
  title={Consistent estimation of the order for Markov and hidden Markov chains},
  author={Finesso, Lorenzo},
  year={1990},
  publisher={University of Maryland, College Park}
}

@book{gilbert1959identifiability,
  title={On the identifiability problem for functions of finite Markov chains},
  author={Gilbert, Edgar J},
  year={1959},
  publisher={Sandia Corporation}
}

@article{self1987asymptotic,
  title={Asymptotic properties of maximum likelihood estimators and likelihood ratio tests under nonstandard conditions},
  author={Self, Steven G and Liang, Kung-Yee},
  journal={Journal of the American Statistical Association},
  volume={82},
  number={398},
  pages={605--610},
  year={1987},
  publisher={Taylor \& Francis}
}

@article{blackwell1957identifiability,
  title={On the identifiability problem for functions of finite Markov chains},
  author={Blackwell, David and Koopmans, Lambert},
  journal={The Annals of Mathematical Statistics},
  pages={1011--1015},
  year={1957},
  publisher={JSTOR}
}

@article{petrie1969probabilistic,
  title={Probabilistic functions of finite state Markov chains},
  author={Petrie, Ted},
  journal={The Annals of Mathematical Statistics},
  volume={40},
  number={1},
  pages={97--115},
  year={1969},
  publisher={JSTOR}
}

@article{anderson2004model,
  title={Model selection and multi-model inference},
  author={Anderson, David and Burnham, Kenneth},
  journal={Second. NY: Springer-Verlag},
  volume={63},
  number={2020},
  pages={10},
  year={2004}
}

@article{douc2005non,
  title={Non singularity of the asymptotic Fisher information matrix in hidden Markov models},
  author={Douc, Randal},
  journal={arXiv preprint math/0511631},
  year={2005}
}

@article{saloff2007convergence,
  title={Convergence of some time inhomogeneous Markov chains via spectral techniques},
  author={Saloff-Coste, Laurent and Z{\'u}niga, Jessica},
  journal={Stochastic processes and their applications},
  volume={117},
  number={8},
  pages={961--979},
  year={2007},
  publisher={Elsevier}
}

@book{hall2014martingale,
  title={Martingale limit theory and its application},
  author={Hall, Peter and Heyde, Christopher C},
  year={2014},
  publisher={Academic press}
}

@article{HEINZEL2025,
title = {Consistency and central limit results for the maximum likelihood estimator in the Admixture Model},
journal = {Theoretical Population Biology},
year = {2025},
issn = {0040-5809},
doi = {https://doi.org/10.1016/j.tpb.2025.10.001},
url = {https://www.sciencedirect.com/science/article/pii/S0040580925000632},
author = {Carola Sophia Heinzel}
}

@article{garcia2020evaluation,
  title={Evaluation of model fit of inferred admixture proportions},
  author={Garcia-Erill, Gen{\'\i}s and Albrechtsen, Anders},
  journal={Molecular ecology resources},
  volume={20},
  number={4},
  pages={936--949},
  year={2020},
  publisher={Wiley Online Library}
}

@article{wang2019parsimony,
  title={A parsimony estimator of the number of populations from a STRUCTURE-like analysis},
  author={Wang, Jinliang},
  journal={Molecular Ecology Resources},
  volume={19},
  number={4},
  pages={970--981},
  year={2019},
  publisher={Wiley Online Library}
}

@article{mackay2002estimating,
  title={Estimating the order of a hidden Markov model},
  author={Mackay, Rachel J},
  journal={Canadian Journal of Statistics},
  volume={30},
  number={4},
  pages={573--589},
  year={2002},
  publisher={Wiley Online Library}
}

@article{bickel1998asymptotic,
  title={Asymptotic normality of the maximum-likelihood estimator for general hidden Markov models},
  author={Bickel, Peter J and Ritov, Ya’acov and Ryden, Tobias},
  journal={The Annals of Statistics},
  volume={26},
  number={4},
  pages={1614--1635},
  year={1998},
  publisher={Institute of Mathematical Statistics}
}

@article{leroux1992maximum,
  title={Maximum-likelihood estimation for hidden Markov models},
  author={Leroux, Brian G},
  journal={Stochastic processes and their applications},
  volume={40},
  number={1},
  pages={127--143},
  year={1992},
  publisher={Elsevier}
}

@article{le2000basic,
  title={Basic properties of the projective product with application to products of column-allowable nonnegative matrices},
  author={Le Gland, Franc{\c{c}}ois and Mevel, Laurent},
  journal={Mathematics of Control, Signals and Systems},
  volume={13},
  pages={41--62},
  year={2000},
  publisher={Springer}
}

@article{le2000exponential,
  title={Exponential forgetting and geometric ergodicity in hidden Markov models},
  author={Le Gland, Franc{\c{c}}ois and Mevel, Laurent},
  journal={Mathematics of Control, Signals and Systems},
  volume={13},
  pages={63--93},
  year={2000},
  publisher={Springer}
}

@article{baum1966statistical,
  title={Statistical inference for probabilistic functions of finite state Markov chains},
  author={Baum, Leonard E and Petrie, Ted},
  journal={The annals of mathematical statistics},
  volume={37},
  number={6},
  pages={1554--1563},
  year={1966},
  publisher={JSTOR}
}

@article{douc2001asymptotics,
  title={Asymptotics of the maximum likelihood estimator for general hidden Markov models},
  author={Douc, Randal and Matias, Catherine},
  year={2001}
}

@article{douc2004asymptotic,
  title={Asymptotic properties of the maximum likelihood estimator in autoregressive models with Markov regime},
  author={Douc, Randal and Moulines, Eric and Ryd{\'e}n, Tobias},
  year={2004}
}

@article{genon2006leroux,
  title={Leroux's method for general hidden Markov models},
  author={Genon-Catalot, Valentine and Laredo, Catherine},
  journal={Stochastic processes and their applications},
  volume={116},
  number={2},
  pages={222--243},
  year={2006},
  publisher={Elsevier}
}

@article{carstens2022assessing,
  title={Assessing model adequacy leads to more robust phylogeographic inference},
  author={Carstens, Bryan C and Smith, Megan L and Duckett, Drew J and Fonseca, Emanuel M and Thom{\'e}, M Tereza C},
  journal={Trends in Ecology \& Evolution},
  volume={37},
  number={5},
  pages={402--410},
  year={2022},
  publisher={Elsevier}
}

@article{tedeschi2006assessment,
  title={Assessment of the adequacy of mathematical models},
  author={Tedeschi, Luis Orlindo},
  journal={Agricultural systems},
  volume={89},
  number={2-3},
  pages={225--247},
  year={2006},
  publisher={Elsevier}
}

@article{Lawley1956,
  author  = {Lawley, D. N.},
  title   = {A General Method for Approximating to the Distribution of Likelihood Ratio Criteria},
  journal = {Biometrika},
  year    = {1956},
  volume  = {43},
  number  = {3-4},
  pages   = {295--303},
  doi     = {10.1093/biomet/43.3-4.295}
}

@article{Hayakawa1977,
  author  = {Hayakawa, Takesi},
  title   = {The Likelihood Ratio Criterion and the Asymptotic Expansion of its Distribution},
  journal = {Annals of the Institute of Statistical Mathematics},
  year    = {1977},
  volume  = {29},
  number  = {1},
  pages   = {359--378},
  doi     = {10.1007/BF02532797}
}

@article{Cordeiro1987,
  author  = {Cordeiro, Gauss M.},
  title   = {On the Corrections to the Likelihood Ratio Statistics},
  journal = {Biometrika},
  year    = {1987},
  volume  = {74},
  number  = {2},
  pages   = {265--274},
  doi     = {10.1093/biomet/74.2.265}
}

@book{cappe2005inference,
  title={Inference in hidden Markov models},
  author={Capp{\'e}, Olivier and Moulines, Eric and Ryd{\'e}n, Tobias},
  year={2005},
  publisher={Springer}
}

@article{sason2016f,
  title={$ f $-divergence Inequalities},
  author={Sason, Igal and Verd{\'u}, Sergio},
  journal={IEEE Transactions on Information Theory},
  volume={62},
  number={11},
  pages={5973--6006},
  year={2016},
  publisher={IEEE}
}

@incollection{hall1992principles,
  title={Principles of edgeworth expansion},
  author={Hall, Peter},
  booktitle={The bootstrap and edgeworth expansion},
  pages={39--81},
  year={1992},
  publisher={Springer}
}

@article{mimno2015posterior,
  title={Posterior predictive checks to quantify lack-of-fit in admixture models of latent population structure},
  author={Mimno, David and Blei, David M and Engelhardt, Barbara E},
  journal={Proceedings of the National Academy of Sciences},
  volume={112},
  number={26},
  pages={E3441--E3450},
  year={2015},
  publisher={National Academy of Sciences}
}

@article{dean2014parameter,
  title={Parameter estimation for hidden Markov models with intractable likelihoods},
  author={Dean, Thomas A and Singh, Sumeetpal S and Jasra, Ajay and Peters, Gareth W},
  journal={Scandinavian Journal of Statistics},
  volume={41},
  number={4},
  pages={970--987},
  year={2014},
  publisher={Wiley Online Library}
}

@article{baumdicker2022efficient,
  title={Efficient ancestry and mutation simulation with msprime 1.0},
  author={Baumdicker, Franz and Bisschop, Gertjan and Goldstein, Daniel and Gower, Graham and Ragsdale, Aaron P and Tsambos, Georgia and Zhu, Sha and Eldon, Bjarki and Ellerman, E Castedo and Galloway, Jared G and others},
  journal={Genetics},
  volume={220},
  number={3},
  pages={iyab229},
  year={2022},
  publisher={Oxford University Press}
}

@article{ephraim2002hidden,
  title={Hidden markov processes},
  author={Ephraim, Yariv and Merhav, Neri},
  journal={IEEE Transactions on information theory},
  volume={48},
  number={6},
  pages={1518--1569},
  year={2002},
  publisher={IEEE}
}

@book{cappé2005inference,
  title={Inference in Hidden Markov Models},
  author={Capp{\'e}, O. and Moulines, E. and Ryd{\'e}n, T.},
  isbn={9780387402642},
  lccn={2005923551},
  series={Springer Series in Statistics},
  url={https://books.google.de/books?id=-3_A3_l1yssC},
  year={2005},
  publisher={Springer}
}

@article{van2008hidden,
  title={Hidden markov models},
  author={van Handel, Ramon},
  journal={Unpublished lecture notes},
  year={2008}
}

@article{kaeuffer2007detecting,
  title={Detecting population structure using STRUCTURE software: effect of background linkage disequilibrium},
  author={Kaeuffer, R and R{\'e}ale, D and Coltman, DW and Pontier, D},
  journal={Heredity},
  volume={99},
  number={4},
  pages={374--380},
  year={2007},
  publisher={Nature Publishing Group}
}

@article{hill1968linkage,
  title={Linkage disequilibrium in finite populations},
  author={Hill, WG and Robertson, Alan},
  journal={Theoretical and applied genetics},
  volume={38},
  pages={226--231},
  year={1968},
  publisher={Springer}
}

@article{jensen1999asymptotic,
  title={Asymptotic normality of the maximum likelihood estimator in state space models},
  author={Jensen, Jens Ledet and Petersen, Niels V{\ae}ver},
  journal={The Annals of Statistics},
  volume={27},
  number={2},
  pages={514--535},
  year={1999},
  publisher={Institute of Mathematical Statistics}
}

@article{brouste2010asymptotic,
  title={Asymptotic properties of MLE for partially observed fractional diffusion system},
  author={Brouste, Alexandre and Kleptsyna, Marina},
  journal={Statistical Inference for Stochastic Processes},
  volume={13},
  pages={1--13},
  year={2010},
  publisher={Springer}
}

@article{choi2018comparison,
  title={Comparison of phasing strategies for whole human genomes},
  author={Choi, Yongwook and Chan, Agnes P and Kirkness, Ewen and Telenti, Amalio and Schork, Nicholas J},
  journal={PLoS genetics},
  volume={14},
  number={4},
  pages={e1007308},
  year={2018},
  publisher={Public Library of Science San Francisco, CA USA}
}

@article{amini2014haplotype,
  title={Haplotype-resolved whole-genome sequencing by contiguity-preserving transposition and combinatorial indexing},
  author={Amini, Sasan and Pushkarev, Dmitry and Christiansen, Lena and Kostem, Emrah and Royce, Tom and Turk, Casey and Pignatelli, Natasha and Adey, Andrew and Kitzman, Jacob O and Vijayan, Kandaswamy and others},
  journal={Nature genetics},
  volume={46},
  number={12},
  pages={1343--1349},
  year={2014},
  publisher={Nature Publishing Group US New York}
}

@article{holsinger2009genetics,
  title={Genetics in geographically structured populations: defining, estimating and interpreting F ST},
  author={Holsinger, Kent E and Weir, Bruce S},
  journal={Nature Reviews Genetics},
  volume={10},
  number={9},
  pages={639--650},
  year={2009},
  publisher={Nature Publishing Group UK London}
}

@article{zheng2016haplotyping,
  title={Haplotyping germline and cancer genomes with high-throughput linked-read sequencing},
  author={Zheng, Grace XY and Lau, Billy T and Schnall-Levin, Michael and Jarosz, Mirna and Bell, John M and Hindson, Christopher M and Kyriazopoulou-Panagiotopoulou, Sofia and Masquelier, Donald A and Merrill, Landon and Terry, Jessica M and others},
  journal={Nature biotechnology},
  volume={34},
  number={3},
  pages={303--311},
  year={2016},
  publisher={Nature Publishing Group US New York}
}

@article{duitama2012fosmid,
  title={Fosmid-based whole genome haplotyping of a HapMap trio child: evaluation of Single Individual Haplotyping techniques},
  author={Duitama, Jorge and McEwen, Gayle K and Huebsch, Thomas and Palczewski, Stefanie and Schulz, Sabrina and Verstrepen, Kevin and Suk, Eun-Kyung and Hoehe, Margret R},
  journal={Nucleic acids research},
  volume={40},
  number={5},
  pages={2041--2053},
  year={2012},
  publisher={Oxford University Press}
}

@article{snyder2015haplotype,
  title={Haplotype-resolved genome sequencing: experimental methods and applications},
  author={Snyder, Matthew W and Adey, Andrew and Kitzman, Jacob O and Shendure, Jay},
  journal={Nature Reviews Genetics},
  volume={16},
  number={6},
  pages={344--358},
  year={2015},
  publisher={Nature Publishing Group UK London}
}

@article{loh2016reference,
  title={Reference-based phasing using the Haplotype Reference Consortium panel},
  author={Loh, Po-Ru and Danecek, Petr and Palamara, Pier Francesco and Fuchsberger, Christian and A Reshef, Yakir and K Finucane, Hilary and Schoenherr, Sebastian and Forer, Lukas and McCarthy, Shane and Abecasis, Goncalo R and others},
  journal={Nature genetics},
  volume={48},
  number={11},
  pages={1443--1448},
  year={2016},
  publisher={Nature Publishing Group US New York}
}

@article{delaneau2012linear,
  title={A linear complexity phasing method for thousands of genomes},
  author={Delaneau, Olivier and Marchini, Jonathan and Zagury, Jean-Fran{\c{c}}ois},
  journal={Nature methods},
  volume={9},
  number={2},
  pages={179--181},
  year={2012},
  publisher={Nature Publishing Group US New York}
}
\end{small}

\section{Supplementary Material}

\subsection{Description of the Code}

\subsubsection{Calculating the Maximum Likelihood Estimators}\label{sec:identify}

 {
To calculate the test statistic, we have to calculate two maximum likelihood estimators: the one in the Linkage Model and the one in the Admixture Model. For the first one, we use random search. Alternatively, we could use grid search, the Expectation-Maximization algorithm or a Bayesian approach with Monte Carlo Markov Chain as implemented in STRUCTURE \citep{pritchard2000}. More precisely, the code performs a derivative-free optimization by sampling many candidate parameters:
\[
q^{(t)} \sim \mathrm{Dirichlet}(\eta \mathbf{1}),\qquad
r^{(t)} \sim \mathrm{LogUniform}(r_{\min},r_{\max}),\qquad t=1,\dots,T,
\]
computes $\ell\!\left(q^{(t)}, r^{(t)}\right)$ via the scaled forward algorithm, and returns
\[
(\hat q,\hat r) = \arg\max_{t\in\{1,\dots,T\}} \ \ell\!\left(q^{(t)}, r^{(t)}\right).
\]
This is an approximate MLE: it is the best among the sampled candidates. To calculate the MLE in the Admixture Model, we use a method that has been introduced by \cite{pfaffelhuber2022}, Lemma S1.1, that relies on a fix point equation. 
We evaluated the performance of our maximization method for the linkage model. Therefore, we simulated 1000 times according to the linkage model for $C = 1, K = 3, M = 50, 100, 200, 500, 1000, 10000$ with $p_{k,m} \sim_{i.i.d.} \mathcal U([0,1]), r^0 \mathcal U([0,10]), q \sim Dirichlet((1,1,1)).$ The means squared error between $(q^0, r^0)$ and the MLE is represented in Figure \ref{fig:eval_MSE}. We see that the performance for estimating the $q$ is much better than for estimating the $r.$ One reason is that the values of $r$ are higher.}

\begin{figure}
    \centering
    \includegraphics[width=0.5\linewidth]{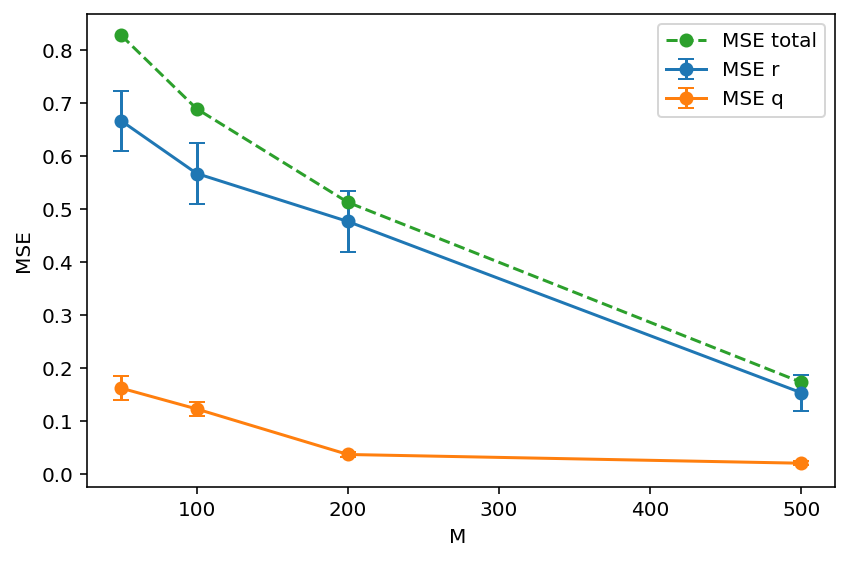}
    \caption{Evaluation of the performance of calculating the MLE in the Linkage Model for $K = 3$. }
    \label{fig:eval_MSE}
\end{figure}

\subsubsection{Simulation with msprime}\label{sec:msprime}

 {
We simulate a neutral, symmetric island model with 
$K$ populations (demes), each having the same effective population size $N_e$ with the help of msprime \cite{baumdicker2022efficient}.
We have a constant pairwise migration rate $m$.
A rough target $F_{ST}$ can be converted into a migration rate via the
(island-model) approximation
\[
F_{ST} \approx \frac{1}{1 + 4N_e (K-1)m}
\quad\Rightarrow\quad
m \approx \frac{1/F_{ST}-1}{4N_e(K-1)}.
\]
For each deme, we sample $n$ diploid individuals.
Although individuals are diploid, \texttt{msprime} represents samples as haploid
nodes; allele frequencies are therefore computed on haplotypes (which is
standard for estimating population allele frequencies). Ancestry is simulated along a chromosome of length $L$ with recombination rate
$\rho$ (per bp per generation). Mutations are then overlaid at rate $\mu$
(per bp per generation).
To obtain multiple independent chromosomes, the code runs $C$ independent
replicates. Across all simulated chromosomes, the code iterates through variants and exclusively considers biallelic sites only.
For each retained SNP, the alternate-allele frequency in deme $k$ is computed as
\[
\hat p_k = \frac{\sum_{i\in k} g_i}{n_k},
\]
where $g_i\in\{0,1\}$ is the haploid genotype at that SNP and $n_k$ is the number
of sampled haplotypes in deme $k$.
The SNP is scored by the range of allele frequencies across demes
\[
\text{score} = \max_k(\hat p_k) - \min_k(\hat p_k).
\]
The top $M$ SNPs by this score are retained using a min-heap. Let $M$ be the number of selected markers. The allele-frequency matrix
$p\in\mathbb{R}^{K\times M}$ is built as $p_{k,t}=\hat p_k$ for marker $t$.
To avoid $\log(0)$ in likelihoods, entries are clipped to
$[\varepsilon,1-\varepsilon]$ for a small $\varepsilon$. Inter-marker distances $d_t$ (in centiMorgans) are approximated from physical
distances using
\[
d_t \approx 100\cdot \rho \cdot \Delta bp_t,
\]
where $\Delta bp_t$ is the distance to the previous selected marker on the same
chromosome. Chromosome boundaries are encoded as a very large distance to force an effective break in linkage. The resulting $(p,d)$ are then used to simulate diploid genotypes under (i) the admixture model and the linkage model,
followed by re-running the statistical test to assess type-I error and power.}
\subsubsection{Extended Information about the simulated data with msprime}

Figure \ref{fig:d1} shows the distribution ob the distances between the markers for the simulated markers with msprime. We see that in all three cases ($M = 50, 100, 1000)$ the distances are rather small (usually smaller than $10^{-4}$, which leads to a high power of the statistical test. 

\begin{figure}[H]
    \centering
    \begin{minipage}{0.45\textwidth}
    \centering
                \textbf{(A)}\\
    \includegraphics[width= \linewidth]{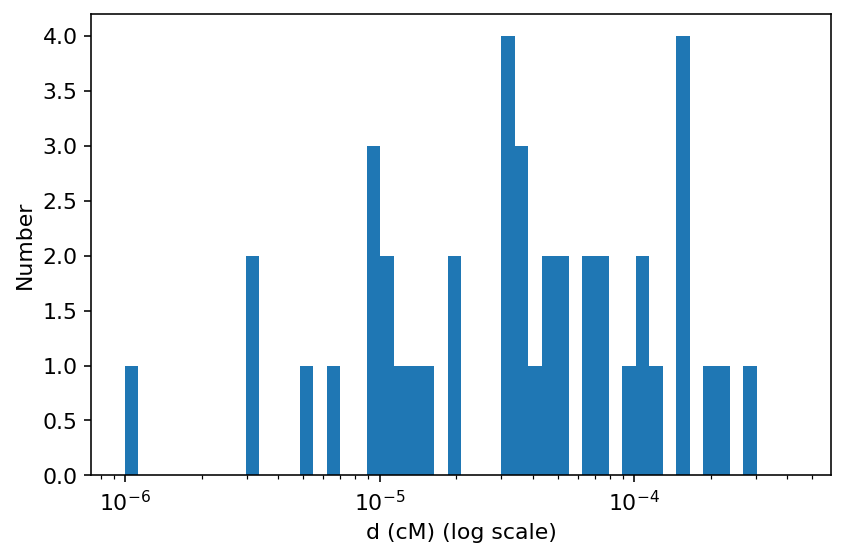}
    \end{minipage}
    \hfill
    \begin{minipage}{0.45\textwidth}
    \centering
                \textbf{(B)}\\
    \includegraphics[width=\linewidth]{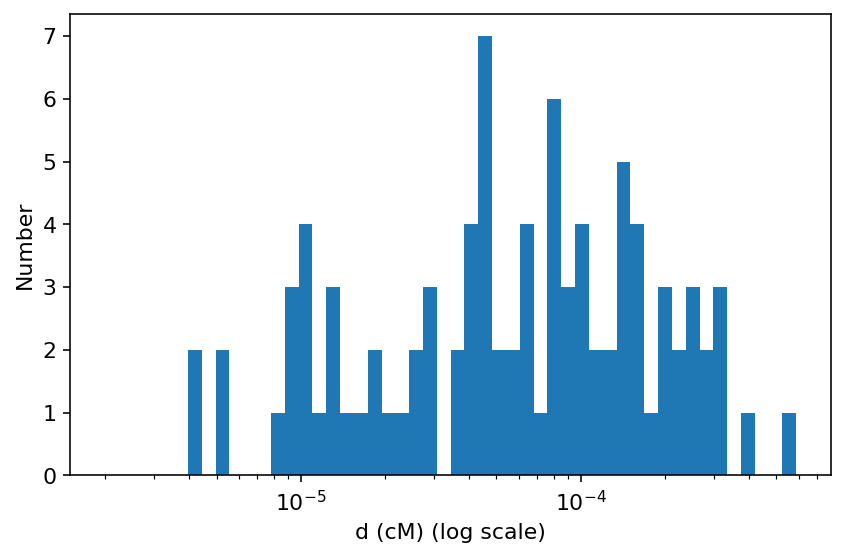}
    \end{minipage}\\
    \begin{minipage}{0.45\textwidth}
    \centering
                \textbf{(C)}\\
    \includegraphics[width= \linewidth]{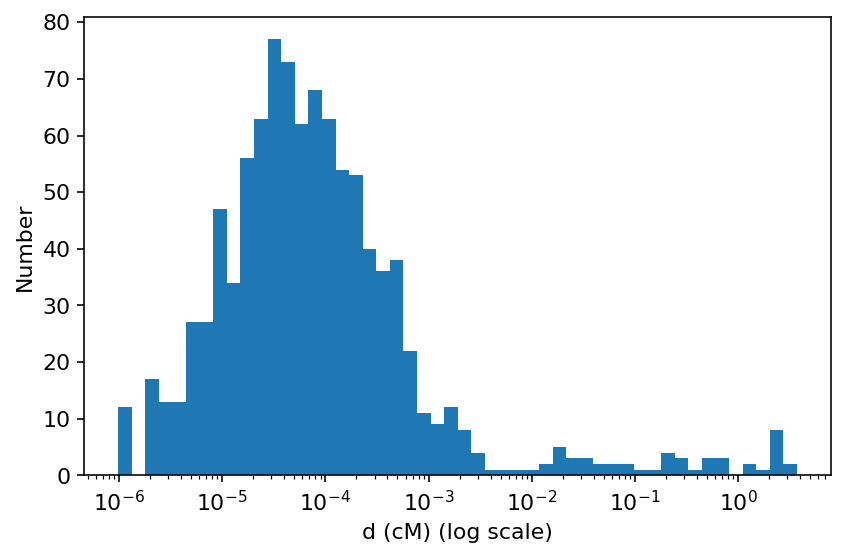}
    \end{minipage}\\
    \caption{Distribution of the distances $d_m$ in the simulation scenarios of msprime (A) $M =  50$, (B) $M = 100$, (C) $M = 1000$.}
    \label{fig:d1}
\end{figure}

\subsection{Extended Information about the real data}

The values for $d$ for the 1000 Genomes data in cM are shown in Figure \ref{fig:distance}.

\begin{figure}[h!]
    \centering
\includegraphics[width=0.95\linewidth]{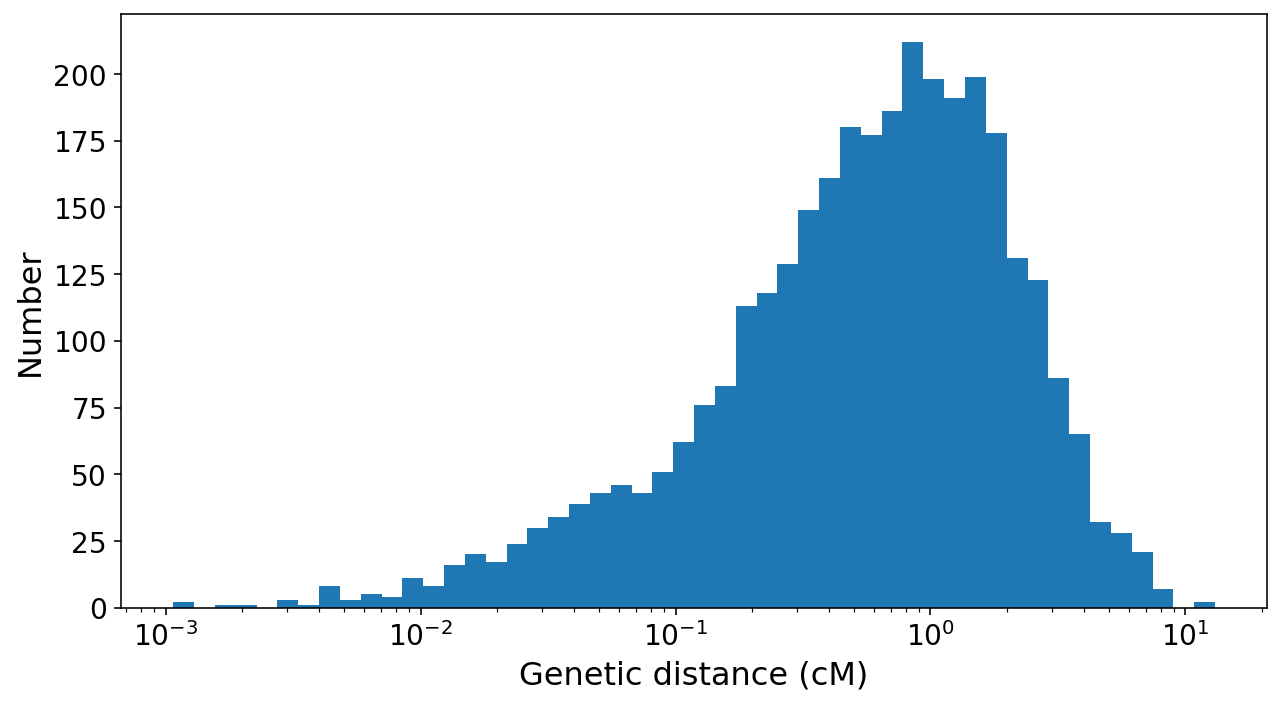}
    \caption{Genetic Distances of the markers that have been used in the \cite{10002015global} data.
}\label{fig:distance}
\end{figure}

\subsection{Estimation of the Uncertainty of the MLEs}

We further compare the covariance matrices of the MLEs obtained under the Linkage Model and the Admixture Model for individual HG00096 from Great Britain. For this analysis, we used 350 randomly chosen markers. The null hypothesis got rejected at significance level 0.05 (p-value: 0.01907). The results are shown in Figure \ref{fig:AM} and Figure \ref{fig:LM}, corresponding to the Admixture and Linkage Models, respectively.  {We accessed the covariance by using bootstrap.}

Notable, the MLEs for $q$ for both models is very similar.

\begin{figure}[H]
    \centering
    \begin{minipage}[t]{0.48\textwidth}
        \centering
        \includegraphics[width=\linewidth]{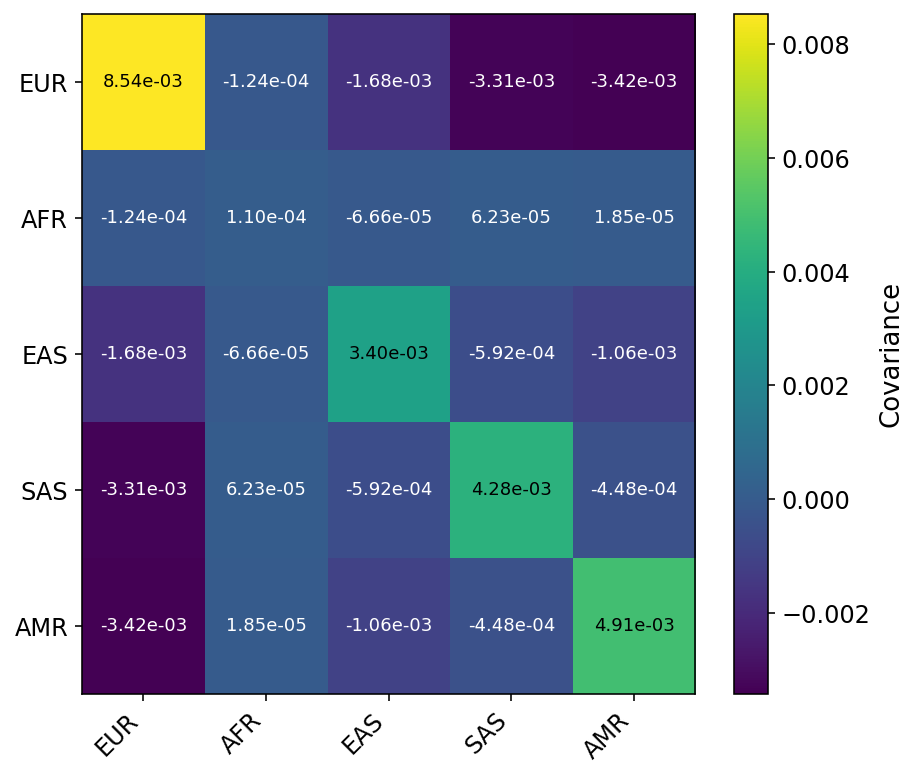}
        \caption{Covariance Matrix for the MLE in the Admixture Model. We considered individual HG00096. The MLE for $q$ was $(0.88 , 0.0047, 0.033, 0.040, 0.042).$ }
        \label{fig:AM}
    \end{minipage}
    \hfill
    \begin{minipage}[t]{0.48\textwidth}
        \centering
        \includegraphics[width=\linewidth]{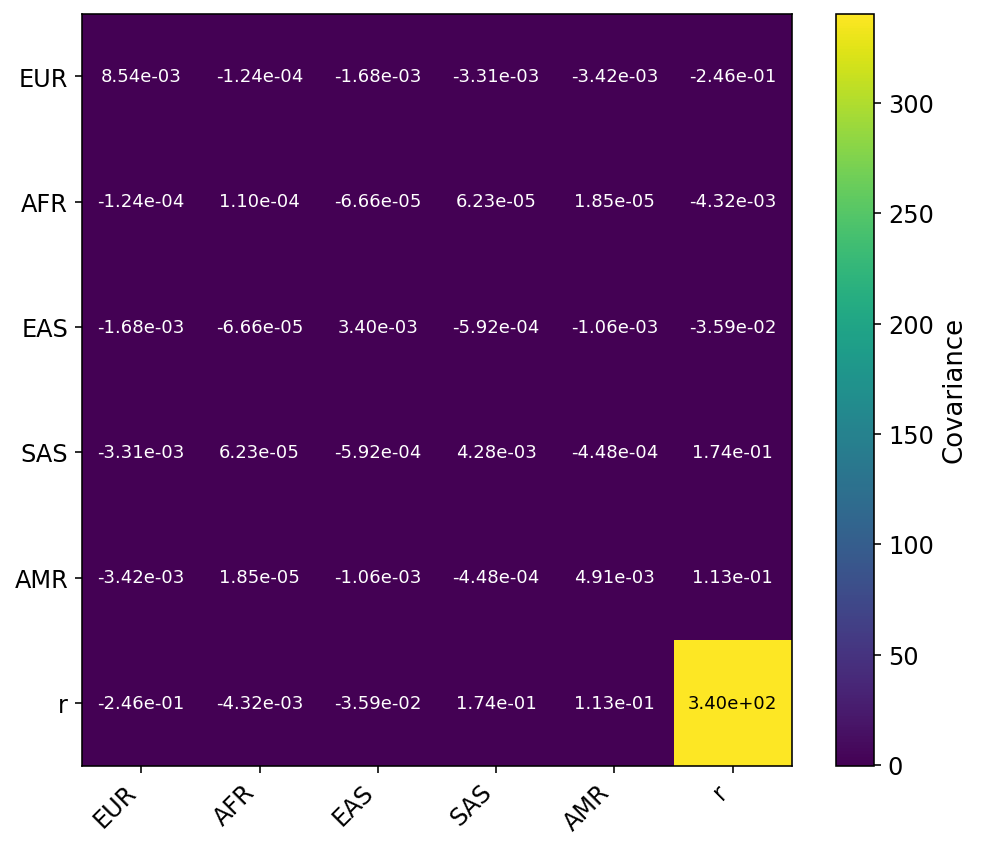}
        \caption{Covariance Matrix of the MLE in the Linkage Model. We considered individual HG00096. The MLE for $q$ was $(0.87, 0.0047 , 0.047, 0.035, 0.038)$ and for $r$ approximately $11.14.$}
        \label{fig:LM}
    \end{minipage}
\end{figure}

As expected, the variance of the estimators is higher in the Linkage Model due to the larger number of parameters. While both models clearly identify EUR as the primary ancestral population, the estimated admixture proportions differ significantly between models (estimated ancestry of almost 1 from Europe in the Admixture Model and estimated ancestry of approximately 0.8 from Europe in the Linkage Model).

\subsection{Evaluation of the statistical test for $K = 5$ and diploid individuals}

 {We also evaluated the performance of the statistical test for diploid individuals and $K = 5.$ The results are shown in Figure \ref{fig:evaluation_K5_diploid}. They, again, show that the Type I error is always smaller than the significance level 0.05. Additionally, we also see that for large $rd,$ test has a small power.}

\begin{figure}[H]
    \begin{minipage}{0.45\textwidth}
        \centering
        \textbf{(A)}\\
        \includegraphics[width=\linewidth]{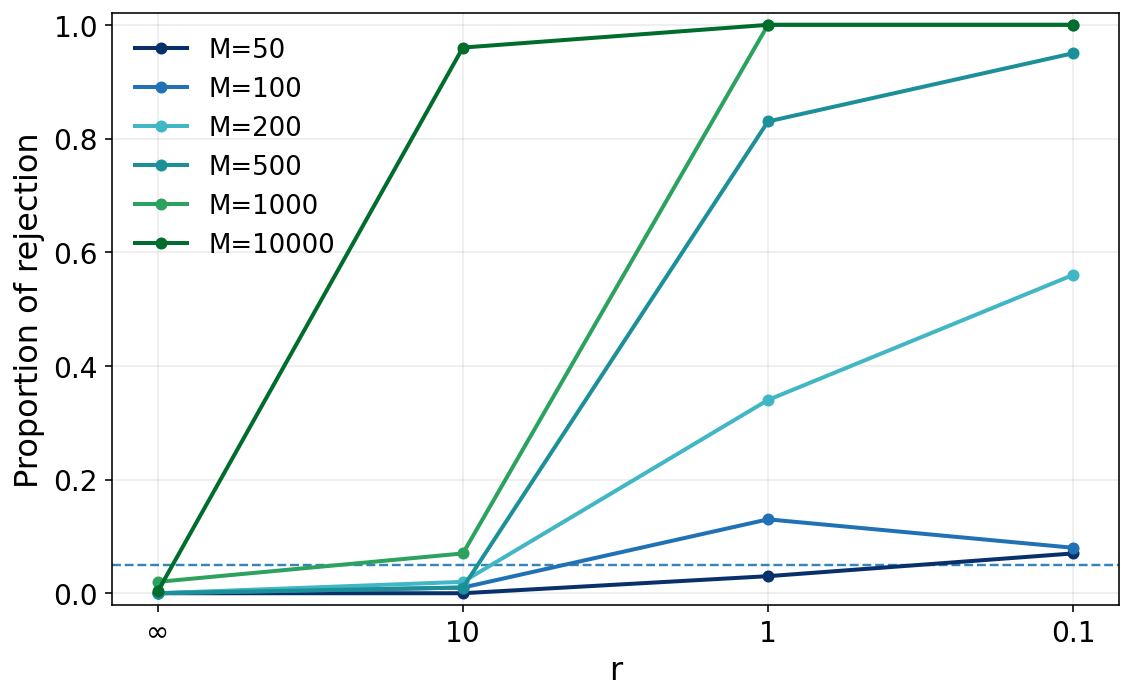}
    \end{minipage}
    \hfill
    \begin{minipage}{0.45\textwidth}
        \centering
        \textbf{(B)}\\
        \includegraphics[width=\linewidth]{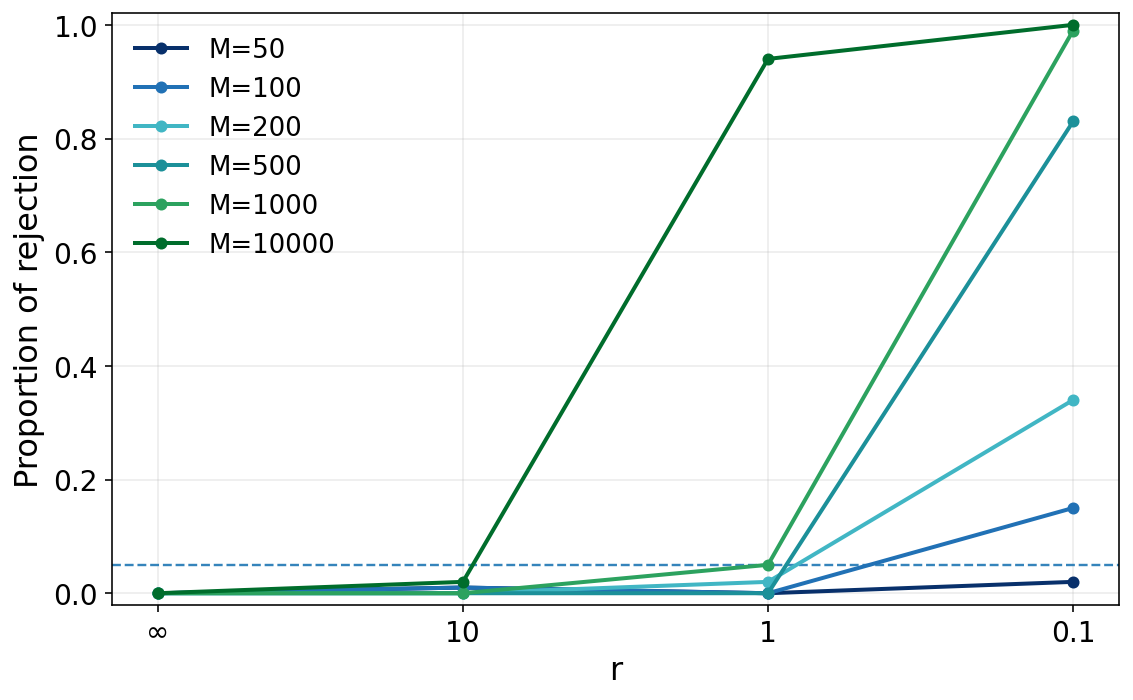}
    \end{minipage}\\[0.8em]

    \begin{minipage}{0.45\textwidth}
        \centering
        \textbf{(C)}\\
        \includegraphics[width=\linewidth]{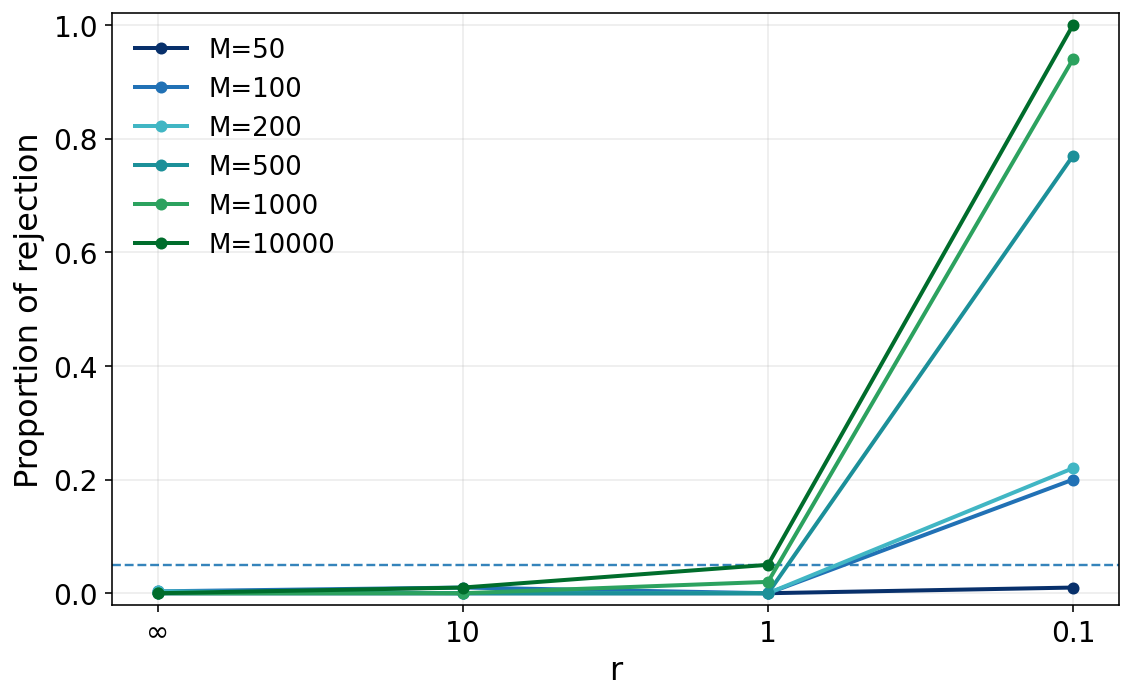}
    \end{minipage}
    \hfill
    \begin{minipage}{0.45\textwidth}
        \centering
        \textbf{(D)}\\
        \includegraphics[width=\linewidth]{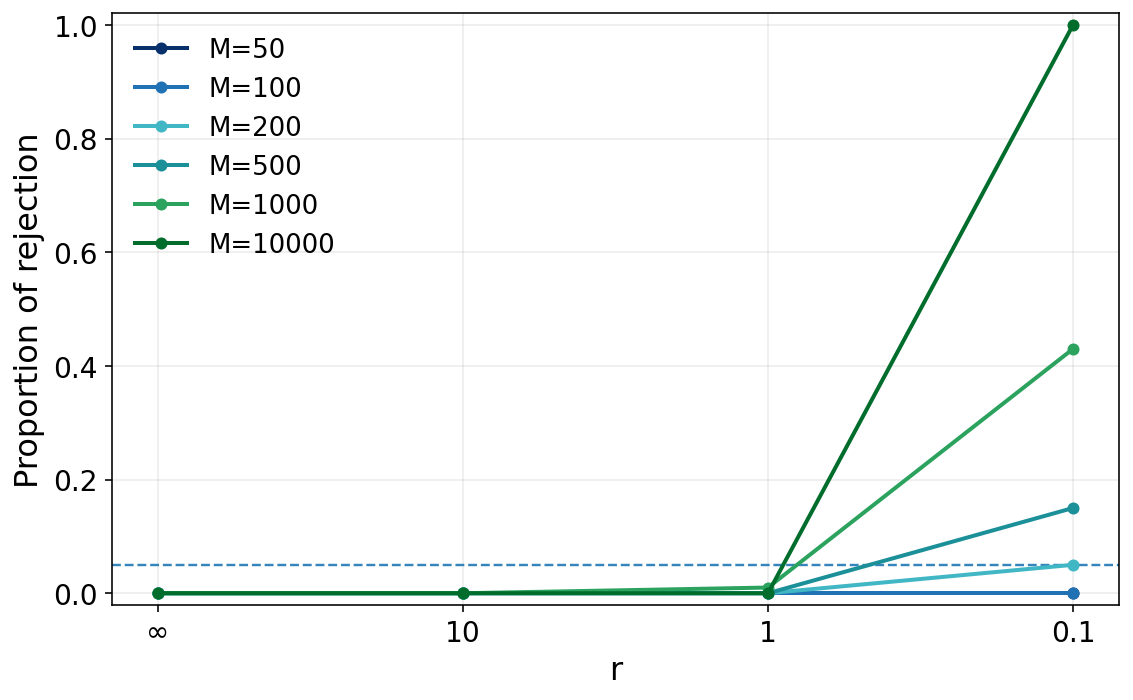}
    \end{minipage}
    \caption{Evaluation of the statistical test by using simulated data for different values of $r$ and $d$ for $K = 5$. We simulated diploid individuals. The true ancestry was $q^0 = (0.2, 0.05, 0.25, 0.4, 0.1).$}
    \label{fig:evaluation_K5_diploid}
\end{figure}

\end{document}